\documentclass[a4paper,12pt]{amsart}
\usepackage{fullpage}

\usepackage{amsmath, amssymb, amsthm,enumitem,bm,pict2e,xcolor,tikz-cd,indentfirst,xparse,ytableau,stackengine,bbm,thmtools}
\usepackage[foot]{amsaddr}
\usepackage[backref=page]{hyperref}

\usepackage{tikz-cd}
\usepackage{bm}

\hypersetup{
    colorlinks=true,
    linkcolor=blue,
    filecolor=magenta,      
    urlcolor=cyan,
}

\declaretheoremstyle[ eröffnen
  headfont=\color{red}\normalfont\bfseries,
  bodyfont=\color{red}\normalfont,
]{red}

\theoremstyle{definition}
\newtheorem{theorem}{Theorem}
\newtheorem*{theorem*}{Theorem}
\numberwithin{theorem}{section}
\newtheorem{definition}[theorem]{Definition}
\newtheorem*{definition*}{Definition}
\newtheorem{proposition}[theorem]{Proposition}
\newtheorem{lemma}[theorem]{Lemma}
\newtheorem{remark}[theorem]{Remark}
\newtheorem{example}[theorem]{Example}
\newtheorem{cor}[theorem]{Corollary}

\newtheorem{innercustomthm}{Theorem}
\newenvironment{customthm}[1]
  {\renewcommand\theinnercustomthm{#1}\innercustomthm}
  {\endinnercustomthm}

\DeclareMathOperator{\grdim}{grdim}
\DeclareMathOperator{\Gr}{Gr}
\DeclareMathOperator{\init}{in}

\DeclareMathOperator{\Spec}{Spec}

\DeclareMathOperator{\Proj}{Proj}
\DeclareMathOperator{\conv}{conv}
\DeclareMathOperator{\aff}{aff}

\DeclareMathOperator{\Trop}{Trop}
\DeclareMathOperator{\Dr}{Dr}
\DeclareMathOperator{\rk}{rk}
\DeclareMathOperator{\cone}{cone}

\newcommand{\cA}{\mathcal A}
\newcommand{\cB}{\mathcal B}

\newcommand{\cD}{\mathcal D}
\newcommand{\cE}{\mathcal E}
\newcommand{\cG}{\mathcal G}

\newcommand{\cI}{\mathcal I}

\newcommand{\cL}{\mathcal L}

\newcommand{\cR}{\mathcal R}

\newcommand{\cZ}{\mathcal Z}

\newcommand{\bC}{\mathbb{C}}

\newcommand{\bP}{\mathbb{P}}
\newcommand{\bR}{\mathbb{R}}
\newcommand{\bZ}{\mathbb{Z}}

\newcommand{\one}{\textbf{\fontfamily{qcs}\selectfont 1}}
\newcommand{\al}{\langle}
\newcommand{\ar}{\rangle}
\newcommand{\bs}{\backslash}
\newcommand{\wt}{\widetilde}

\newcommand{\bfK}{\mathbf{K}}
\newcommand{\bfC}{\mathbf{C}}
\newcommand{\bfJ}{\mathbf{J}}

\newcommand{\suchthat}{\, : \,}

\newcommand{\op}{\mathrm{op}}

\DeclareMathOperator{\Sec}{Sec}
\DeclareMathOperator{\subd}{subd}
\DeclareMathOperator{\Grob}{Gr\ddot{o}b}
\DeclareMathOperator{\Aff}{Aff}

\DeclareFontFamily{U}{mathx}{\hyphenchar\font45}
\DeclareFontShape{U}{mathx}{m}{n}{
      <5> <6> <7> <8> <9> <10>
      <10.95> <12> <14.4> <17.28> <20.74> <24.88>
      mathx10
      }{}
\DeclareSymbolFont{mathx}{U}{mathx}{m}{n}
\DeclareMathSymbol{\bigtimes}{1}{mathx}{"91}

\mathcode`l="8000
\begingroup
\makeatletter
\lccode`\~=`\l
\DeclareMathSymbol{\lsb@l}{\mathalpha}{letters}{`l}
\lowercase{\gdef~{\ifnum\the\mathgroup=\m@ne \ell \else \lsb@l \fi}}%
\endgroup

\mathchardef\newbracket=\mathcode`)
\mathcode`)="8000
\begingroup
\catcode`) \active
\gdef){\nolinebreak\newbracket}
\endgroup

\mathchardef\newcomma=\mathcode`,
\mathcode`,="8000
\begingroup
\catcode`, \active
\gdef,{\nolinebreak\newcomma}
\endgroup

\date{May 2025}


\title[Regular subdivisions and initial ideals]{Regular subdivisions, bounds on initial ideals,\\ and categorical limits}

\author{George Balla}
\address[GB]{Max Planck Institute for Mathematics in the Sciences, Leipzig}
\email[GB]{george.balla@mis.mpg.de}

\author{Daniel Corey}
\address[DC]{Department of Mathematical Sciences, University of Nevada, Las Vegas}
\email[DC]{daniel.corey@unlv.edu}

\author{Igor Makhlin}
\address[IM]{Technische Universit\"at Berlin}
\email[IM]{iymakhlin@gmail.com}

\author{Victoria Schleis}
\address[VS]{Institute for Advanced Study, Princeton, and Durham University, United Kingdom}
\email[VS]{victoria.m.schleis@durham.ac.uk}

\begin{document}

\begin{abstract}
Several known constructions relate initial degenerations of projective toric varieties and Grassmannians to regular subdivisions of appropriate point configurations. We define a general framework which allows for partial generalizations of these constructions to arbitrary projective schemes (as well as their very affine parts). We associate a point configuration $A$ with any homogeneous ideal $I$. We obtain upper and lower bounds on every initial ideal of $I$, defining them in terms of the regular subdivision of $A$ given by the same weight. Furthermore, both bounds are interpreted categorically via (co)limits over the face poset of the subdivision. We also investigate when these bounds are exact, showing that the respective weights form a subfan in the secondary fan of $A$.
\end{abstract}

\maketitle

\section*{Introduction}

%Given an ideal $I\subset\bC[x_1,\dots,x_m]$, one may associate an initial ideal $\init_w I$ with any weight vector $w\in\bR^m$. 
Initial ideals have their origins in computer algebra, specifically in the work of Buchberger, and play a central role in the area of Gr\"obner theory.
A particularly important construction of initial ideals relies on a choice of a real weight vector \cite{MoraRobbiano, Robbiano}. This construction allows for a key application to algebraic geometry, where initial ideals serve as standard tools for constructing flat degenerations \cite[Section~15.8]{Eisenbud}. 
We refer to~\cite{Stu1} for an introduction to various aspects of initial ideals.

On the other hand, for any real point configuration a real weight vector determines a regular subdivision. 
Defined in its modern form in~\cite{GKZ}, this concept can be traced to the work of Voronoy. It provides a standard method of building point-set triangulations and is widely applied in discrete and computational geometry, algebraic combinatorics and other fields, see \cite{Lafforgue2003, MaclaganSturmfels, Ziegler}.  A detailed introduction can be found in~\cite{DeLoeraRambauSantos}. 

The main goal of this paper is to initiate a general study of connections between initial ideals and regular subdivisions %, which makes the superficial similarities between the two constructions more precise. 
that would give a more systematic approach to several results obtained in special cases, see below.

Let $I\subset \bC[x_1,\ldots,x_m]$ be an ideal and let $A$ be a configuration of $m$ points in a real vector space. As indicated above, a weight vector $w\in \bR^{m}$ is a required input in order to form an initial ideal $\init_w I$ of $I$ and a regular subdivision $\subd_w A$ of $A$. Under mild assumptions, both the set of weight vectors providing a given initial ideal and the set of weight vectors inducing a given regular subdivision form a polyhedral cone. In both cases the collection of all such cones is a complete polyhedral fan in $\bR^{m}$: the Gr\"obner fan $\Grob I$ and the secondary fan $\Sec A$.

%A rather basic observation is that both of the above constructions depend on a weight vector. Further, under mild assumptions, both the set of weight vectors providing a given initial ideal and the set of weight vectors inducing a given regular subdivision form a polyhedral cone. Moreover, in both cases the collection of all such cones is a complete polyhedral fan: the Gr\"obner fan $\Grob I$ and the secondary fan $\Sec A$.

In special settings these observations are known to be deeply connected. First, let $A$ be a lattice point configuration and $I$ be the respective homogeneous toric ideal. In~\cite[Chapter 8]{Stu1} it is shown that $\Grob I$ refines $\Sec A$, i.e. the initial ideal determines the subdivision uniquely. Moreover, results in~\cite{Stu1, Zh} provide an elegant geometric interpretation: the irreducible components of the zero set of $\init_w I$ are the toric varieties of the maximal cells of $\subd_w A$.

Now, let $I_{k,n}$ be the Pl\"ucker ideal defining the Grassmannian $\Gr(k,n)$ and let $A$ consist of the vertices of the hypersimplex $\Delta(k,n)$. A particularly important subfan of $\Grob I_{k,n}$ is the tropical Grassmannian $\Trop I_{k,n}$. An important subfan of  $\Sec A$ is the Dressian $\Dr(k,n)$, parameterizing matroid subdivisions of $\Delta(k,n)$ \cite{HerrmannJensenJoswigSturmfels}. In~\cite{Speyer, Tev} it is shown that every cone of $\Trop I_{k,n}$ is contained in a cone of $\Dr(k,n)$, providing a morphism between the two subfans.  Subsequently,~\cite{Cor} establishes a relationship between the initial degeneration and the matroid subdivision determined by $w\in\Trop I_{k,n}$: the very affine scheme defined by $\init_w I_{k,n}$ admits a closed immersion into a certain finite limit of a diagram determined by the cells of $\subd_w A$.

We provide a general framework that encompasses both of these types of results and realize these two setting as dual to each other. 
%The broad motivation of our project is to search for a unifying and generalizing context for the above results. In this article we present extensions of both of these types of  results 
%to arbitrary projective schemes, also showing that these two settings should be viewed as dual to each other. 
%As a first key step
To begin, we associate a point configuration $\cA(I)$ with any homogeneous ideal $I$. One may view $\cA(I)$ as a projection of the vertices of the unit simplex of $\bR^{m}$ onto the lineality space of $\Grob I$. For the Pl\"ucker ideal this is the vertices of the hypersimplex, and for a homogeneous toric ideal $I$ it is the point configuration defining $I$. 

The first main result we describe concerns bounds on $\init_{w}I$ by ideals associated to regular subdivisions of $\cA(I)$. To the subdivision $\Theta = \subd_{w}\cA(I)$ we associate an ideal $I_{w}$ defined as a sum of ideals over the cells of $\Theta$. Dually, to the subdivision $\Theta^{*} = \subd_{-w}\cA(I)$  we associate an ideal $I^{w}$ defined as an intersection of ideals over the cells of $\Theta^*$.

\begin{customthm}{A}[cf.\ Theorems~\ref{sumcontained} and \ref{thm:intersectioncontains}]\label{thm:A}
    In $\bC[x_1,\dots,x_m]$, we have the inclusion of ideals
    \begin{equation*}
        I_{w} \subseteq \init_{w}I \subseteq I^{w}.
    \end{equation*}
\end{customthm}

\noindent If $I=I_{k,n}$, then the finite limit in \cite{Cor} may be realized as the very affine scheme cut out by the ideal $I_{w}$. If $I$ is toric, then $I^{w}$ is the radical of $\init_{w}I$ as shown in \cite{Zh}. 

Next, with every cell $\Delta$ of $\Theta$ we associate a ring $R_\Delta$.
For faces $\Delta\subset \Gamma$ of $\Theta$, we have an embedding $R_{\Delta} \to R_{\Gamma}$. Dually, we associate rings $R^\Delta$ with cells $\Delta$ of $\Theta^*$ and obtain surjections $R^{\Gamma}\to R^{\Delta}$ for $\Delta\subset \Gamma$. Thus, the assignments $\Delta\mapsto R_{\Delta}$ and $\Delta\mapsto R^{\Delta}$ define (category-theoretical) diagrams $\cR$ and $\cR^{*}$ indexed by the face posets of $\bfJ(\Theta)$ and $\bfJ(\Theta^{*})$ of $\Theta$ and $\Theta^{*}$ respectively. We identify the limits of these diagrams to be the quotient rings of $I_w$ and $I^w$ respectively. Set $R_{w}=\bC[x_1,\dots,x_m]/I_w$ and $R^{w}=\bC[x_1,\dots,x_m]/I^w$. 

\begin{customthm}{B}[cf.\ Theorems~\ref {thm:colimit-hat} and~\ref{thm:limit}]
    We have canonical isomorphisms
    \begin{equation*}
        R_{w} \cong \varinjlim_{\bfJ(\Theta)} R_{\Delta} 
        \quad \text{and} \quad 
        R^{w} \cong \varprojlim_{\bfJ(\Theta^*)^\mathrm{op}} R^{\Delta}.
    \end{equation*}
\end{customthm}

The ideals $I_{w}$ and $I^{w}$ are algorithmically easier  to compute than $\init_{w} I$, and a natural question is to identify when the inclusions in Theorem~\ref{thm:A} are equalities. 
We define
\begin{equation*}
    \Omega(I) = \{w\in \bR^{m}  \suchthat I_{w} = \init_{w}I\} 
    \quad \text{and} \quad
    \Omega^{*}(I) =  \{w\in \bR^{m}  \suchthat I^{w} = \init_{w}I\}.
\end{equation*}
These sets may be used to quantify the accuracy of the bounds of $\init_w I$ provided by 
$I_{w}$ and $I^{w}$, respectively. 
Further, we prove a structural result about these sets. 

\begin{customthm}{C}[cf. Theorems~\ref{issubfan} and~\ref{dualissubfan}]
    The set $\Omega(I)$ is the support of a subfan of $\Sec(A)$, and $\Omega^{*}(I)$ is the support of a subfan of $-\Sec(A)$. 
\end{customthm}

Finally, we extend some of these results to the setting of very affine schemes. Suppose that the ideal $I$ contains no monomials  and hence defines a very affine scheme.  We associate to each cell $\Delta$ of $\Theta$ a very affine scheme $X_{\Delta}^{\circ}$. Then, for faces $\Delta\subset \Gamma$ we have morphisms $X_{\Gamma}^{\circ} \to X_{\Delta}^{\circ}$. The assignment $\Delta\mapsto X_{\Delta}^{\circ}$ defines a finite diagram of very affine schemes and we may form their limit $\varprojlim_{\bfJ(\Theta)^\mathrm{op}}X_{\Delta}^{\circ}$ in the category of affine schemes. Next, we establish a connection to initial degenerations by defining natural morphism of schemes
\begin{equation}
\label{eq:very-affine-inw-limit}
    \init_{w} X^{\circ} \to \varprojlim_{\bfJ(\Theta)^\mathrm{op}}X_{\Delta}^{\circ}.
\end{equation}
which, under mild hypotheses, is a closed immersion (cf.\ Theorem \ref{thm:very-affine-initial-deg}). This generalizes the main theorem in \cite{Cor}. Under these same mild hypotheses, we also show that the limit in (1) can be realized as the very affine scheme defined by $I_w$ (cf.\ Theorem~\ref{thm: limit is Xw}).

% we connect the limit back to the ideals $I_w.$
% Denote by $S^{\circ} = \bC[x_1^{\pm}, \ldots, x_{m}^{\pm}]$ and by $I_w^\circ$ the extension of $I_w$ to $S^{\circ}$. 
% We write $R_w^\circ$ for the $\bC^\times$-invariant subring of $S^\circ/I_w^\circ$ and set $X_w^\circ=\Spec R_w^\circ$. We obtain Theorem \ref{thm: limit is Xw}, which asserts
% \begin{equation*}
%     \varprojlim_{\bfJ(\Theta)^\mathrm{op}} X_{\Delta}^{\circ}=X_w^\circ.
% \end{equation*}

%In  Appendix \ref{appendix}, we conclude by applying our theory back to the case of Grassmannians. Here, we consider weights $w\in \Trop I_{k,n}$, which yield all  $X_{\Delta}^{\circ}$ to be matroid strata. While the closed immersion $\init_{w} X^{\circ} \to \varprojlim_{\bfJ(\Theta)^\mathrm{op}}X_{\Delta}^{\circ}$ is only an isomorphism for $k = 2$ and for $k =3$ and $n\leq 8$ in general \cite{Cor,CoreyLuber}, we identify a special class of weights $w$ where the map is also an isomorphism. This class is constructed by considering inductively connected matroids. We further show that the initial degenerations with respect to these $w$ are smooth and irreducible.

In Appendix \ref{appendix}, we consider again the Grassmannian. % describe a new infinite collection of initial degenerations of Grassmannians that are smooth and irreducible. 
In general, initial degenerations of the very affine $\Gr(k,n)^{\circ}$ may be reducible or nonsmooth, and the closed immersion in Equation \eqref{eq:very-affine-inw-limit} might not be an isomorphism \cite{Cor, CoreyLuber-SMRS}.  
We identify a new infinite collection of initial degenerations $\init_w\Gr(k,n)^{\circ}$ that are smooth and irreducible, and where the above closed immersion is an isomorphism. This collection corresponds to matroids that are realizable over $\bC$, paving, connected, and inductively connected as defined in  \cite{LiwskiMohammadi}.

\subsection*{Computations}

In Appendix \ref{appendixB}, we obtain an algorithm for computing the point configuration $\cA(I)$ associated with a given homogeneous ideal $I$.
We implement this algorithm using the OSCAR package \cite{OSCAR-book, OSCAR} for Julia \cite{Julia-2017} and create functions that compute $I_w$, $I^{w}$, $\Omega(I)$, and $\Omega^*(I)$. 
The code can be found at the following github repository: 
%\medskip
\begin{equation}\label{githublink}\tag{$*$}
\text{\url{https://github.com/dmg-lab/InitialIdealsRegularSubdivisions/}}
\end{equation}

\subsection*{Acknowledgments}
G.B., D.C. and V.S. were partially supported by SFB-TRR project ``Symbolic Tools in Mathematics and their Application'' (project- ID 286237555). 
During a part of this project, V.S. was a member of the Institute for Advanced Study, funded by the Charles Simonyi Endowment. Further, V.S. was partially funded by the UKRI FLF \emph{Computational Tropical Geometry}  MR/Y003888/1.

\section{Point configurations and their subdivisions}

In the following we define a correspondence between point configurations and vector subspaces that plays an important role in our results. We then briefly recall the construction of the secondary fan of a point configuration and interpret the associated subspace as the lineality space of this fan.

\subsection{Point configurations and subspaces}\label{subspaceconfig}

For the entirety of this paper we fix a finite labeling set $E$. 
An \textit{$E$-labeled point configuration} in the $\bR$-vector space $V$
%A \textit{point configuration} in an $\bR$-vector space $V$ labeled by $E$ (i.e.\ \textit{$E$-labeled}) 
is a sequence of points $A = (a_e)_{e\in E}$ in $V$. We denote by $\aff A$ the affine hull of $A$,  the dimension of $A$ is $\dim A=\dim(\aff A)$.
Two point configurations $A=(a_e)_{e\in E}$ and $B=(b_e)_{e\in E}$ are \textit{affinely equivalent} if there is an affine bijection $\varphi\colon \aff A\to \aff B$ 
such that $\varphi(a_e)=b_e$ for all $e\in E$. Denote by $[A]_{\aff}$ the 
class of all point configurations affinely equivalent to $A$. 
%The dimension of an affine point configuration is defined as $\dim A=\dim(\aff A)$.

\begin{definition}
Given a point configuration $A=(a_e)_{e\in E}$ in the $\bR$-vector space $V$, define the following linear subspace of $\bR^{E}$:
\begin{multline*}
    \cL(A) = \{w \in \bR^{E}  \suchthat \text{there exists an affine function } f\colon V\to \bR\text{ such that }\\
    w_e = f(a_{e}) \text{ for all }e\in E  \}.
\end{multline*}
%which is a linear subspace of $\bR^{E}$. 
%Consider an $E$-labeled point configuration $A=(a_e)_{e\in E}$ in a real space $V$. We denote by $\cL(A)$ the vector subspace in $\bR^E$ consisting of those points $w=(w_e)_{e\in E}$ for which there is an affine function $f\colon V\to \bR$ such that $w_e=f(a_e)$ for all $e\in E$.
\end{definition} Evidently, $\dim\cL(A)=\dim A+1$ and $\cL(A)$ contains the all-ones vector  $\one = (1,\dots,1)$. Furthermore, if $A$ and $B$ are affinely equivalent, then $\cL(A)=\cL(B)$. In fact, we show that the assignment $[A]_{\aff}\mapsto \cL(A)$ defines a bijection between affine equivalence classes and subspaces containing $\one$. The inverse bijection is obtained as follows. Let $\{u_{e}\}_{e\in E}$ be the standard basis of $(\bR^{E})^*$. 
\begin{definition}
%%For $e\in E$ let $f_e\in(\bR^E)^*$ denote the functional mapping a point $w$ to its coordinate $w_e$. 
Given a linear subspace $L\subset\bR^E$, let $u^{L}_{e}\in L^*$ denote the restriction of $u_{e}$ to $L$. Define the following $E$-labeled point configuration in $L^*$:
\begin{equation*}
    \cA(L) = (u^{L}_{e})_{e\in E}.
\end{equation*}
%For a subspace $L\subset\bR^E$, denote by  $\cA(L)$ the $E$-labeled point configuration in $L^*$ consisting of restrictions of the $f_e$ to $L$, where $(f_e)_{e\in E} \subset (\bR^{E})^*$ is the dual basis to the standard basis of $\bR^{E}$.  
%%Its \textit{character polytope} $\cP(L)\subset L^*$ is the convex hull of $\cA(L)$.
\end{definition}

We now establish the aforementioned bijection.
\begin{proposition}\label{bijection}
The map $L\mapsto[\cA(L)]_{\aff}$ is a bijection from the set of subspaces of $\bR^E$ containing $\one$ to the set of affine equivalence classes of $E$-labeled point configurations. 
The inverse bijection takes $[A]_{\aff}$ to $\cL(A)$.
%, i.e.\ one has $\cL(\cA(L))=L$ for any $L$ containing $\one$ and $[\cA(\cL(A))]_{\aff}=[A]_{\aff}$ for any $E$-labeled point configuration $A$.
\end{proposition}
\begin{proof}
Consider a linear subspace $L$ containing $\one$. The linear function on $L^*$ given by $w\in L$ takes value $w_e$ in $u_e^L$. Therefore, $w\in\cL(\cA(L))$ and $L\subset\cL(\cA(L))$. However, all $u_e^L$ take value 1 in $\one$, hence $\dim\cA(L)\le\dim L-1$ and $\dim\cL(\cA(L))\le\dim L$. We deduce that $L=\cL(\cA(L))$.  

Conversely, for a point configuration $A$, the space $\cL(A)$ can be viewed as the space of affine functions on $\aff A$. We obtain an embedding of $\aff A$ into $\cL(A)^*$ that maps $w$ to the functional taking value $f(w)$ in an affine function $f\in\cL(A)$. By definition, this embedding maps $a_e\in A$ to $u_e^{\cL(A)}$. Hence, $A$ is affinely equivalent to $\cA(\cL(A))$.
\end{proof}

%{\color{cyan} perhaps remove here to end of Cor \ref{cor:Grassmannian-to-affine}.}
We give two alternative interpretations of this correspondence. First, we write $\Aff(k, E)$ for the set of affine equivalence classes of $E$-labeled point configurations of dimension $k$. Let $W$ denote the quotient $\bR^E/\bR\one$ by the diagonal and let $q:\bR^E\to W$ be the quotient map. For $k\in[0,|E|-1]$, consider the Grassmannian $\Gr(k,W)$ of $k$-subspaces in $W$. Proposition~\ref{bijection} can now be restated as follows.
\begin{cor}
\label{cor:Grassmannian-to-affine}
   We have a bijection between $\Gr(k,W)$ and $\Aff(k,E)$ that takes $q(L)$ to $[\cA(L)]_{\aff}$ for every $(k+1)$--subspace $L\subset\bR^E$ containing $\one$. 
\end{cor}

Second, we describe an explicit method of obtaining a point configuration affinely equivalent to $\cA(L)$. For a subspace $L\subset\bR^E$ this is done by projecting the coordinate vectors orthogonally onto $L$. Indeed, the standard scalar product on $\bR^E$ induces a linear isomorphism $\iota_L\colon L\to L^*$. Let $\pi_L: \bR^E\to L$ denote the orthogonal projection onto $L$ and let $(\varepsilon_e)_{e\in E}$ be the standard basis in $\bR^E$.
\begin{proposition}\label{orthogonalproj}
The isomorphism $\iota_L$ takes the point configuration $(\pi_L(\varepsilon_e))_{e\in E}$ to $\cA(L)$. In particular, the two point configurations are affinely equivalent.
\end{proposition}
\begin{proof}
The scalar product of $\pi_L(\varepsilon_e)$ with $w\in L$ is equal to $w_e$, hence $\iota_L(\pi_L(\varepsilon_e))=u_e^L$.    
\end{proof}

%\begin{remark}
%    We provide a  matrix representation of $\cA(L)$ which may be useful for concrete computations. Let $k = \dim(L)$. Let $M$ be a $|E|\times k$ matrix whose rows are labeled by $E$ and whose columns form an ordered basis of $L$. Then the rows of $M$ form an $E$-labeled point configuration in the affine equivalence class of $\cA(L)$. 
%\end{remark}

\subsection{Regular subdivisions and secondary fans}

A comprehensive treatment of the following discussion can be found in~\cite{DeLoeraRambauSantos}.

For an $E$-labeled point configuration $A$ in $V$, a \textit{point subconfiguration} of $A$ is any point configuration of the form $A'=(a_e)_{e\in E'}$ with labeling set $E'\subseteq E$; we write $A'\subseteq A$. A \textit{face} of $A$ is a point subconfiguration of $A$ that has the form $A_f$ for some $f\in  V^*$, where
\[A_f=\left(a_{e'} \in A  \suchthat f(a_{e'}) \leq f(a_{e}) \text{ for all } e\in E \right).\]
%consisting of points that minimize a certain functional. Explicitly, a face is any point subconfiguration that has the form $A_f$ for some functional $f\in V^*$ where 
%\[A_f=\left(a_e \in A  \suchthat f(a_e) \leq f(a_{e'}) \text{ for all } e'\in E \right).\]
%{\color{cyan} Isn't $w$ usually viewed in $(\bR^{E})^{*}$?} 
Given $w\in \bR^{E}$, the \textit{lifted point configuration} $A^{w}$ in $ V \oplus \bR$ is $A^{w} = (a_e\oplus w_e)_{e\in E}$.
For $f\in V^{*}$ we may consider the face $A^w_{f\oplus 1}$ of $A^w$ where $f\oplus 1 \in (V\oplus \bR)^*$ evaluates as 
\[(f\oplus 1)(v\oplus w)=f(v)+w.\] 
Point configurations of the form $A^w_{f\oplus 1}$ are also known as the \textit{lower faces} of $A^w$. They should be thought of as those faces of $A^w$ that are visible from points $v\oplus w$ with $w\ll 0$.

Now let $\pi\colon V\oplus \bR \to V$ denote the natural projection.
%Now let $\pi$ denote the natural projection from $V\oplus\bR$ to $V$. 
Then $\Delta_f=\pi(A^w_{f\oplus 1})$ is a point subconfiguration of $A$ for any $f\in V^*$.
%Given two regular subdivisions $\Theta_1$ and $\Theta_2$, we say that $\Theta_1$ \emph{refines} $\Theta_2$ if every cell of $\Theta_1$ is contained in a cell of $\Theta_2$. 
The \textit{regular subdivision} of $A$ induced by $w$, denoted by $\subd_{w}(A)$, is the set of all point subconfigurations of the form $\Delta_f$. 
These point subconfigurations are also known as the \textit{cells} of the regular subdivision. 

The obtained regular subdivision is a polyhedral subdivision (see~\cite[Definition 2.3.1]{DeLoeraRambauSantos}), this has several consequences. First, the intersection of two cells is again a cell. 
%For any cell $\Delta$, the convex hull $\conv(\Delta)$ is a polytope. Using these two facts, we obtain that the cells form a polyhedral subdivision of $A$, and that $$\bigcup_{\Delta_f \in \subd_{w}(A)} \conv(\Delta_f) = \conv(A)$$ 
Furthermore, a point subconfiguration $\Gamma\subset\Delta_f$ is also in $\subd_w(A)$ if and only if $\Gamma$ is a face of $\Delta_f$. 

Next, note that there may be points in $A$ that are not contained in any cell of $\subd_w(A)$. However, for any $a_e\in A$ there exists a unique minimal cell $\Delta$ such that $a_e$ lies in the convex hull $\conv(\Delta)$. Conversely, if $a_e \in \conv(\Delta)$, but $a_e\notin \Delta$, then $a_e$ is not contained in any cell (if $a_e \in \conv(\Delta)$ and $a_e$ is contained in a cell, then it is contained in a face of $\Delta$, and thus in $\Delta$).  

Given a regular subdivision $\Theta$ of $A$, define
\begin{equation*}
    \cone(A, \Theta)^{\circ} = \{w\in \bR^{E}  \suchthat \subd_{w}(A) = \Theta \}, \quad \cone(A, \Theta) = \overline{\cone(A, \Theta)^{\circ}}.
\end{equation*}
The sets $\cone(A, \Theta)^\circ$ are relatively open polyhedral cones in $\bR^{E}$. Given two regular subdivisions $\Theta_1$ and $\Theta_2$, we say that $\Theta_1$ \emph{refines} $\Theta_2$ if every cell of $\Theta_1$ is contained in a cell of $\Theta_2$. Then, $\cone(A, \Theta_2)$ is a face of $\cone(A, \Theta_1)$ if and only if the subdivision $\Theta_1$ refines the subdivision $\Theta_2$. %, again, in a natural sense. 
The \textit{secondary fan} of $A$, denoted by $\Sec A$, is the complete polyhedral fan in $\bR^{E}$ whose cones are the $\cone(A, \Theta)$ as $\Theta$ ranges over all regular subdivisions of $A$. 
%Their closures $\cone(A, \Theta)$, with $\Theta$ ranging over all regular subdivisions of $A$, form a complete polyhedral fan in $\bR^E$: the \textit{secondary fan} of $A$ denoted by $\Sec{A}$.

It is clear that affinely equivalent $E$-labeled point configurations have the same secondary fan. Furthermore, the lineality space of $\Sec{A}$ is the set of those $w$ for which the subdivision $\subd_w(A)$ is trivial, i.e.\ consists of the faces $A$. This happens if and only if there is an affine function $f$ on $\bR^E$ such that $w_e=f(a_e)$ for all $e\in E$. We obtain
\begin{proposition}\label{prop:lineality}
The lineality space of $\Sec A$ is $\cL(A)$.
\end{proposition}

\section{Point configurations associated to ideals}

\subsection{Initial ideals and Gr\"obner fans}\label{sec:initial-ideals-groebner-fans}

We recall the construction of the Gr\"obner fan of a homogeneous ideal. See, e.g.\ \cite[Chapter 1]{Stu1} for details.

Consider the polynomial ring $S=\bC[x_e]_{e\in E}$. A vector $w\in\bR^E$ can be viewed as an $\bR$-grading on $S$ that takes the value $w_e$ on $x_e$. A polynomial $p\in S\bs\{0\}$ is the sum of its homogeneous components with respect to this grading. The \textit{initial form} of $p$, denoted by $\init_w p$, is its lowest (nonzero) homogeneous component, i.e.\ the sum of those terms in $p$ which have the least grading with respect to $w$. For an ideal $I\subset S$ its \textit{initial ideal} $\init_w I$ is the ideal spanned by all $\init_w p$ with $p\in I$. 
%This subspace of $S$ is easily seen to also be an ideal. 
%Occasionally we will also use the more general notion of an \textit{initial subspace}: for a vector subspace $U\subset S$ let $\init_w U$ be the linear span of $\{\init_w p\}_{p\in U}$.

For the remainder of the paper we fix an ideal $I\subset S$ that is homogeneous with respect to the standard grading $\one$ and does not contain the irrelevant ideal. We consider the projective scheme $X = \Proj S/I$.

Initial ideals play an important role in algebraic geometry by providing flat degenerations. Specifically, for any $w\in\bR^E$ there is a flat degeneration of $X$ to $\Proj(S/\init_w I)$: a flat family over $\mathbb A^1$ with the fiber over 0 isomorphic to $\Proj(S/\init_w I)$ and all other fibers isomorphic to $X$. Such a degeneration is known as an \textit{initial degeneration} or a \textit{Gr\"obner degeneration}.
%In algebraic terms, this means that there is a flat $\bC[t]$-algebra $\widetilde S$ such that $\widetilde S/\langle t-a\rangle\simeq R$ for all nonzero $a\in\bC$ while $\widetilde S/\langle t\rangle\simeq S/\init_w I$, see e.g.\ \cite[Corollary 3.2.6]{HH1}. {\color{cyan} can we cut this last sentence? it will also cut the reference.}

Given an initial ideal $J$ of $I$, define
\begin{equation*}
    \cone(I, J)^\circ = \{w\in \bR^{E}  \suchthat \init_w I=J\}, \quad \cone(I,J) = \overline{\cone(I, J)^\circ}.
\end{equation*}
The sets $\cone(I, J)^\circ$ are relatively open polyhedral cones in $\bR^{E}$. The \textit{Gr\"obner fan} of $I$, denoted by $\Grob{I}$, is the complete polyhedral fan in $\bR^{E}$ whose cones are $\cone(I, J)$ as  $J$ ranges over all initial ideals of $I$.  %Their closures $\cone(I, J)$, with $J$ ranging over all initial ideals of $I$, form a complete polyhedral fan in $\bR^E$: the \textit{Gr\"obner fan} of $I$, denoted by $\Grob{I}$.

The set of those $w$ for which $\init_w I$ does not contain monomials is the \textit{tropicalization} of $I$, denoted by $\Trop I$. It is the support of a subfan of $\Grob I$.

\subsection{From ideals to point configurations}

The observations in Section \ref{subspaceconfig} let us associate a point configuration with the ideal $I$. The lineality space of $\Grob{I}$, which we denote $\cL(I)$, consists of all those $w\in\bR$ for which $\init_w I=I$, i.e.\ $I$ is homogeneous with respect to the grading $w$. In particular, $\cL(I)$ contains $\one$. We denote the point configuration $\cA(\cL(I))$ by $\cA(I)$. 

The configuration $\cA(I)$ can also be understood in terms of torus actions.
The torus $(\bC^\times)^E$ acts naturally on $S$ by automorphisms: $t(x_e)=t_ex_e$ for $t\in (\bC^\times)^E$. We have an algebraic subtorus $T\subset (\bC^\times)^E$ consisting of those $t$ for which $t(I)=I$. This subtorus is generated by the one-dimensional subtori $\{(s^{w_e})_{e\in E}\}_{s\in\bC^\times}$ with $w\in\cL(I)\cap\bZ^E$. Here note that since the subspace $\cL(I)\subset\bR^{E}$ is rational, $\bZ^E\cap\cL(I)$ is a lattice of rank $\dim\cL(I)$. This lattice is the cocharacter lattice of $T$. Let $M\subset \cL(I)^*$ denote the dual lattice, i.e.\ the character lattice of $T$. The points of $\cA(I)$ lie in $M$ --- the point labeled by $e$ is the character mapping $t\in T$ to $t_e$. This allows us to view $\cA(I)$ as a lattice point configuration.

%Furthermore, the point configuration $\cA(I)$ is naturally realized in the character lattice of $T$. Note that the subspace $\cL(I)$ is necessarily rational and $\bZ^E\cap\cL(I)$ is a lattice of dimension $\dim\cL(I)$. Let $M\subset \cL(I)^*$ denote the dual lattice, the elements of $\cA(I)$ are seen to be lattice points lying in $M$. However, the lattice $\bZ^E\cap\cL(I)$ is the cocharacter lattice of $T$ and $M$ is its character lattice. Explicitly, the point in $\cA(I)$ labeled by $e$ is the character mapping $t\in T$ to $t_e$.

We now give our two main motivating examples. %of the subspace $\cL(I)$ and the point configuration $\cA(I)$.

\begin{example}\label{toricex}
Let $A=(a_e)_{e\in E}$ be an $E$-labeled point configuration in $\bR^m$ whose points lie in the lattice $\bZ^m$.
%lattice points, i.e.\ elements of $\bZ^m$. 
Let $I_A$ denote the homogeneous toric ideal defined by $A$, i.e.\ the kernel of the homomorphism %from $S$ to $\bC[t,z_1^{\pm 1},\dots,z_m^{\pm 1}]$ given by
\begin{equation*}
    S \to \bC[t,z_1^{\pm 1},\dots,z_m^{\pm 1}],\quad  x_e\mapsto tz_1^{(a_e)_1}\dots z_m^{(a_e)_m}.
\end{equation*}
%It is easily seen that
The ideal $I_A$ is spanned by all binomials $\prod x_e^{\alpha_e}-\prod x_e^{\beta_e}$ with $\sum \alpha_e=\sum \beta_e$ and $\sum \alpha_ea_e=\sum \beta_ea_e$, which implies that $\cL(I_A)=\cL(A)$. Thus, $\cA(I_A)$ is affinely equivalent to $A$ by Proposition~\ref{bijection}. 

Conversely, every homogeneous toric (i.e.\ prime binomial) ideal $I\subset S$ coincides with $I_{\cA(I)}$, the latter defined with respect to any choice of basis in the lattice $M$. Indeed, $I$ arises as the kernel of a monomial map, meaning that $I=I_A$ for some lattice point configuration $A$. By the above, $A$ is affinely equivalent to $\cA(I)$. Hence, $I_{\cA(I)}=I_A$.
\end{example}

\begin{example}\label{hypersimplexex}
Choose integers $1\le k<n$ and let $E$ consist of integer tuples $(i_1,\dots,i_k)$ such that 
\begin{equation*}
1\le i_1<\dots<i_k\le n.
\end{equation*}
The space  $\bC^{E}$ can be identified with the exterior power $\wedge^k\bC^n$,
which lets us consider the Pl\"ucker embedding of the Grassmannian $\Gr(k,n)\hookrightarrow \bP(\bC^E)$. The ring $S$ is the homogeneous coordinate ring of $\bP(\bC^E)$, and the \emph{Pl\"ucker ideal} $I_{k,n}\subset S$ is the vanishing ideal of $\Gr(k,n)$.
%The ring $S$ is the homogeneous coordinate ring of $\bP(\bC^E)$, let $I \subset S$ be the vanishing ideal of the subvariety $\Gr(k,n)$, i.e.\ the \emph{Pl\"ucker ideal}. 
The space $\cL(I_{k,n})\subset\bR^E$ is $n$-dimensional and is spanned by the vectors $w_1,\dots,w_n$, where $(w_i)_e=1$ if the tuple $e\in E$ contains $i$ and $(w_i)_e=0$ otherwise. 

Denote the points in $\cA(I_{k,n})$ by $a_e$ for $e\in E$. Consider the basis in $\cL(I_{k,n})^*$ dual to $w_1,\dots,w_n$. With respect to this dual basis, the point $a_e\in\cL(I_{k,n})^*$ is the indicator vector of the tuple $e$. %meaning that $(a_e)_i=1$ if $e$ contains $i$ and $(a_e)_i=0$ otherwise. 
Thus, $\cA(I_{k,n})$ consists of vertices of the hypersimplex $\Delta(k,n)$.
\end{example}

\section{Bounds on the initial ideal}
\label{sec:bounds-on-initial-ideal}

In this section we fix a point $w\in\bR^E$ and study the initial ideal $\init_w I$. We construct a lower bound on this ideal in terms of the regular subdivision $\subd_w\cA(I)$ and an upper bound in terms of the subdivision $\subd_{-w}\cA(I)$. As before, we write $a_e$ to denote the element of $\cA(I)$ labeled by $e\in E$.

\subsection{The lower bound}
\label{sec:lower-bound}

Let $\Theta$ denote the regular subdivision $\subd_w\cA(I)$. As a first observation we note that $\init_w(I)=I$ if and only if $\Theta$ is trivial as both conditions are equivalent to $w\in\cL(I)$. %That is because both conditions are equivalent to $w\in\cL(I)$.

Now consider any cell $\Delta$ of $\Theta$. Denote $\bC[\Delta]=\bC[x_e]_{a_e\in \Delta}$, we view $\bC[\Delta]$ as a subring of $S$. Let $I_\Delta\subset\bC[\Delta]$ be the image of $I$ under the surjection $\rho_\Delta\colon S\to\bC[\Delta]$ that maps every $x_e$ with $e\in \Delta$ to itself and all other $x_e$ to 0. Note that $I_\Delta$ is an ideal in $\bC[\Delta]$ and is also a subspace of $S$. For two cells one of which is a face of the other, the respective ideals satisfy a particularly simple relation: 
%which plays an important role in our results.

\begin{proposition}\label{intersection}
For cells $\Delta\subset \Gamma$ of $\Theta$ one has $I_\Delta=I_\Gamma\cap\bC[\Delta]$.
\end{proposition}

\begin{proof}
The affine span of $\cA(I)$ is the subspace of functionals in $\cL(I)^*$ taking value $1$ on $\one$ and, thus, does not contain $0$. Since $\Delta$ is necessarily a face of $\Gamma$, we may choose $u\in \cL(I)$ such that for $a\in \Delta$ one has $a(u)=0$ and for $b\in \Gamma\bs\Delta$ one has $b(u)>0$.

Let $\rho$ denote the surjection from $\bC[\Gamma]$ to $\bC[\Delta]$ taking all $x_e$ with $e\in\Gamma\bs\Delta$ to 0. In particular, $\rho(I_\Gamma)=I_\Delta$. We obviously have $I_\Gamma\cap\bC[\Delta]=\rho(I_\Gamma\cap\bC[\Delta])\subset I_\Delta$. For the reverse inclusion, we consider $p\in I_\Gamma$ and check that $\rho(p)\in I_\Gamma$. Indeed, either $\rho(p)=0$ or $\rho(p)=\init_u p$ due to our choice of $u$. But $I_\Gamma$ is homogeneous with respect to $u\in\cL(I)$, hence $\init_u p\in I_\Gamma$ and, therefore, $\rho(p)\in I_\Gamma\cap\bC[\Delta]$ in both cases.
\end{proof}

In particular, this lets us verify the following key fact only for maximal cells, i.e.\ those of dimension $\dim A$.

\begin{proposition}\label{initialcontains}
If $\Delta$ is a cell of $\Theta$, then $I_\Delta\subset\init_w I$.
%$I_\Delta\subset\init_w I$ for any cell $\Delta$ of $\Theta$.   
\end{proposition}

\begin{proof}
By Proposition~\ref{intersection} we may assume $\Delta$ is maximal, i.e.\ $\conv\Delta$ has dimension $\dim \cL(I)-1$. We may choose $w'\in\bR^E$ such that $w-w'\in \cL(I)$ and $w'_e=0$ for all $e\in \Delta$. Note that $\subd_{w'}\cA(I)=\Theta$ and that $w'_e>0$ for all $e\notin \Delta$. Also note that $\init_{w'} I=\init_{w} I$. Now consider $p\in I_\Delta$, suppose $p=\rho_\Delta(q)$ for $q\in I$. If $p\neq 0$, we must have $p=\init_{w'}(q)\in \init_{w} I$.
\end{proof}

For a cell $\Delta$ of $\Theta$ let $\wt I_\Delta$ denote the ideal in $S$ generated by $I_\Delta\subset\bC[\Delta]\subset S$. We set 
\begin{equation}\label{eq:I_wdef}
I_w=\sum_{\Delta\text{ cell of }\Theta}\wt I_\Delta.
\end{equation}
%note that here
In other words, $I_w$ is the ideal in $S$ generated by the subspaces $I_\Delta$. By Proposition \ref{intersection} it suffices to sum over maximal cells in~\eqref{eq:I_wdef}.  Proposition~\ref{initialcontains} gives the following lower bound on the initial ideal. 

\begin{theorem}\label{sumcontained}
The initial ideal $\init_w I$ contains $I_w$.
\end{theorem} 

This theorem provides a surjection from $S/I_w$ to $S/\init_w I$, which geometrically can be interpreted as an upper bound on the initial degeneration.

\begin{cor}
There is a closed immersion of $\Proj(S/\init_w I)$ into $\Proj(S/I_w)$.
\end{cor}

\begin{example}\label{prinipaltoricex}
Let $E=\{1,2,3,4\}$ and $I=\al x_1x_2-x_3x_4\ar$. The subspace $\cL(I)$ is cut out in $\bR^E$ by the equation $w_1+w_2=w_3+w_4$. To determine the point configuration $\cA(I)$ up to affine equivalence, we apply Proposition~\ref{orthogonalproj}. The orthogonal projections $\pi_{\cL(I)}(\varepsilon_i)$ are the vertices of a square with the first and second point diagonally opposite each other. There are two nontrivial regular subdivisions of $\cA(I)$. One of them is $\Theta_1 = \subd_w\cA(I)$ for $w$ satisfying  $w_1+w_2<w_3+w_4$. It has maximal two maximal cells: $\Delta_3=(a_1,a_2,a_3)$ and $\Delta_4=(a_1,a_2,a_4)$. We see that in this case $I_{\Delta_3}$ and $I_{\Delta_4}$ are the ideals generated by $x_1x_2$ in the respective rings. We obtain 
$I_w=\al x_1x_2\ar=\init_w I$.
The other nontrivial subdivision $\Theta_2$ is given by $w$ with $w_1+w_2>w_3+w_4$ and also has two maximal cells. In this case we similarly obtain $I_w=\init_w I=\al x_3x_4\ar$. Consequently, for the chosen ideal $I$ the bound in Theorem~\ref{sumcontained} is exact for any $w\in\bR^E$.
\end{example}
\begin{example}\label{doublepointex}
A simple example of a proper inclusion $I_w\subset\init_wI$ is given by $E=\{1,2\}$, the ideal $I=\al x_1x_2(x_1-x_2)\ar$ and $w$ with $w_1\neq w_2$. In this case $\cA(I)$ consists of two coinciding points but $\Theta$ has only one cell consisting of a single point. This leads to $I_w=0$.
\end{example}

We now consider the setting of Example~\ref{hypersimplexex} in the special case $k=2$, i.e.\ the Pl\"ucker ideal $I_{2,n}$ defining $\Gr(2,n)$. The following result can be deduced using the explicit description of $\Trop I_{2,n}$ given in~\cite{SS} and the description of the corresponding subdivisions of $\Delta(2,n)$ given in~\cite{Kapranov1993}. It shows that the lower bound provided by Theorem~\ref{sumcontained} can be exact in a nontrivial setting.
\begin{theorem}
\label{Gr2nsame}
If $I=I_{2,n}$ and $w\in\Trop I$, then $I_w=\init_w I$.
\end{theorem}
\begin{proof}
Consider $w\in\Trop I$. By the results of \cite{SS}, there is a unique \textit{phylogenetic tree} $T$ corresponding to $w$. This is a tree with leaves labeled by $\{1, \ldots, n\}$, no vertices of degree 2, and a real weight attached to every edge which is positive for non-leaf edges. The coordinate $w_{i,j}$ is equal to the sum of weights of edges lying on the (acyclic) path between leaves $i$ and $j$. 

Now, recall that $I$ is generated by the Pl\"ucker relations $p_{i,j,k,l}=x_{i,j}x_{k,l}-x_{i,k}x_{j,l} + x_{i,l}x_{j,k}$ for all quadruples $1\le i<j<k<l\le n$. By~\cite{SS}, the initial forms $\init_w p_{i,j,k,l}$ generate $\init_w I$. The above characterization of $T$ allows us to express $\init_w p_{i,j,k,l}$ as follows. Consider the minimal subtree of $T$ that contains the leaves $i$, $j$, $k$ and $l$. This subtree has four leaves with all remaining vertices having degree two except for either 
\begin{enumerate}[label=(\roman*)]
\item one vertex of degree four or
\item two vertices of degree three. 
\end{enumerate}
If (i) one holds, then $\init_w p_{i,j,k,l}=p_{i,j,k,l}$. If (ii) holds, then $\init_w p_{i,j,k,l}$ is a binomial that
contains the monomial $x_{i',j'}x_{k',l'}$ (where $\{i',j',k',l'\}=\{i,j,k,l\}$) if and only if the path between leaves $i'$ and $j'$ and the path between leaves $k'$ and $l'$ intersect in at least one vertex.

As discussed in Example~\ref{hypersimplexex}, the point configuration $\cA(I)$ is formed by the vertices of the hypersimplex $\Delta(2,n)$. The results of Kapranov (\cite[Theorem 1.3.6]{Kapranov1993}) provide the following description of $\Theta=\subd_w\cA(I)$. The maximal cells correspond to non-leaf vertices of $T$. The cell $\Delta_v$ corresponding to a non-leaf vertex $v$ consists of those points $a_{i,j}$ for which the path between leaves $i$ and $j$ passes through $v$. 
% (Such pairs $\{i,j\}$ form the basis set of a matroid and we obtain a matroid subdivision of $\Delta(2,n)$.)

To show that $I_w=\init_w I$, we check that every $\init_w p_{i,j,k,l}$ lies in $I_{\Delta_v}$ for some $v$. Indeed, if (i) holds, we may choose $v$ as the vertex of degree four. If (ii) holds, we may choose $v$ as either of the two vertices of degree three. In both cases we will have 
\[\init_w p_{i,j,k,l}=\rho_{\Delta_v}(p_{i,j,k,l})\in I_{\Delta_v}.\qedhere\]

% Then for all quadruples $i,j,k,l$, the point $w$ satisfies the following four-point condition:
% \begin{equation}\label{eqn:four-point}
%     w_{ij} + w_{kl} = w_{ik} + w_{jl} < w_{il} + w_{jk}.
% \end{equation}
% Recall that the ideal $I$ is generated by the quadrics $x_{ij}x_{kl}-x_{ik}x_{jl} + x_{il}x_{jk}$ for all quadruples $i,j,k,l$ such that $1\leq 1<i<j<k<l\leq n$. For a fixed trivalent tree $T$, the corresponding initial ideal $\init_w I_{2,n}$ is generated by the binomials $x_{ij}x_{kl}-x_{ik}x_{jl}$. 
% However, the Condition \eqref{eqn:four-point} means that $\{\{i,l\},\{j,k\}\}$ is a four-leaf subtree of $T$. Such a subtree has two internal vertices. For each subtree $\{\{i,l\},\{j,k\}\}$, we therefore have two maximal cells with $\{i,l\}$ and $\{j,k\}$ respectively as the non-bases. It follows that $I_w$ coincides with $\init_w I_{2,n}$ for every subtree $\{\{i,l\},\{j,k\}\}$ of $T$. This shows that $\init_w I\subseteq I_w$.
\end{proof}

\subsection{The upper bound}\label{sec:upper-bound}

Let $\Theta^*$ denote the regular subdivision $\subd_{-w}\cA(I)$. In other words, $\Theta^*$ is the subdivision given by the \textit{upper} faces of the lifted point configuration $A^w$. For a cell $\Delta$ of $\Theta^*$ we consider the ideal $I^\Delta=I\cap\bC[\Delta]$ in $\bC[\Delta]$, the latter ring defined as above. We prove a series of statements concerning these ideals that should be viewed as dual to the statements in the previous subsection. For cells $\Delta\subset \Gamma$ of $\Theta^*$ let $\rho_\Delta^\Gamma$ denote the projection from $\bC[\Gamma]$ to $\bC[\Delta]$ taking $x_e\in\bC[\Delta]$ to itself and all other variables to zero.

\begin{proposition}\label{prop:projection}
For cells $\Delta\subset \Gamma$ of $\Theta^*$ one has $\rho_\Delta^\Gamma(I^\Gamma)=I^\Delta$. 
\end{proposition}
\begin{proof}
Evidently, $I^\Delta=\rho_\Delta^\Gamma(I^\Delta)\subset\rho_\Delta^\Gamma(I_\Gamma)$. We show that $\rho_\Delta^\Gamma(p)\in I^\Delta$ for $p\in I^\Gamma$. As in the proof of Proposition~\ref{intersection}, we may consider $u\in\cL(I)$ that vanishes on $\Delta$ and takes positive values in other points of $\Gamma$. 
Then, either $\rho_\Delta^\Gamma(p)=0$ or $\rho_\Delta^\Gamma(p)=\init_u p$. Since $\init_u I=I$, we have $\rho_\Delta^\Gamma(p)\in I^\Delta$ in both cases.
\end{proof}

For a cell $\Delta$ of $\Theta^*$, let $\wt I^\Delta$ denote the ideal in $S$ generated by the subspace $I^\Delta$ and all $x_e$ with $e\notin\Delta$. In other words, $\wt I^\Delta$ is the preimage of $I^\Delta$ under the surjection $\rho_\Delta$, the latter defined as in the previous subsection. Proposition~\ref{prop:projection} implies that $\wt I^\Gamma\subset\wt I^\Delta$ when $\Delta\subset\Gamma$. 
Indeed, an element $p\in I^\Gamma$ obviously lies in the space $\rho_\Delta^\Gamma(p)+\al x_e\ar_{e\notin\Delta}$. However, $\al x_e\ar_{e\notin\Gamma}\subset\wt I^\Delta$, which together with the above proposition provides $\rho_\Delta^\Gamma(p)+\al x_e\ar_{e\notin\Delta}\subset\wt I^\Delta$. 
%This, again, allows us to prove the next claim  only for maximal cells $\Delta$.

\begin{proposition}\label{prop:initialcontained}
If $\Delta$ is a cell of $\Theta^{*}$, then $\init_w I\subset\wt I^\Delta$.
\end{proposition}
\begin{proof}
In view of the above, it suffices to consider the case of maximal $\Delta$. Similarly to the proof of Proposition~\ref{initialcontains}, we choose $w'\in w+\cL(I)$ such that $w'_e=0$ for $e\in\Delta$ and $w'_e<0$ (note the sign) otherwise. We have $\init_{w'}I=\init_w I$, for $p\in I$ we check that $\init_{w'}p\in\wt I^\Delta$. Indeed, if $p\in\bC[\Delta]$, then $\init_{w'} p=p\in\wt I^\Delta$. Otherwise, $\init_{w'} p\in\al x_e\ar_{e\notin\Delta}\subset\wt I^\Delta$.
\end{proof}

We now set \[I^w=\bigcap_{\Delta\text{ cell of }\Theta^*} \wt I^\Delta.\] As we have seen, it suffices to intersect over maximal $\Delta$. Proposition~\ref{prop:initialcontained} gives an upper bound on the initial ideal. 
\begin{theorem}\label{thm:intersectioncontains}
The initial ideal $\init_w I$ is contained in $I^w$.
\end{theorem}

The theorem provides a surjection from $S/\init_w I$ to $S/I^w$ and its geometric counterpart is a lower bound on the initial degeneration.
\begin{cor}
There is a closed immersion of $\Proj(S/I^w)$ into $\Proj(S/\init_w I)$.
\end{cor}

\begin{example}\label{ex:upperbound}
Consider the setting of Example~\ref{prinipaltoricex}. If $w_1+w_2>w_3+w_4$, then $\subd_{-w}\cA(I)$ has two maximal cells: the same two point configurations $\Delta_3$ and $\Delta_4$. One sees that both $I^{\Delta_3}$ and $I^{\Delta_4}$ are zero. However, $\wt I^{\Delta_3}=\al x_4\ar$ and $\wt I^{\Delta_4}=\al x_3\ar$ and, therefore, $I^w=\al x_3x_4\ar=\init_w I$. Similarly, for $w_1+w_2<w_3+w_4$ one has $I^w=\al x_1x_2\ar=\init_w I$. Thus, for the chosen $I$ the bound in Theorem~\ref{thm:intersectioncontains} is also exact for all $w$.
\end{example}

In the above example $I$ is toric. The restriction of the construction in this section to toric ideals is particularly well-behaved, this is essentially the setting discussed in~\cite{Zh}. As seen in Example~\ref{toricex}, if $I$ is toric, it can be identified with the toric ideal defined by the lattice point configuration $\cA(I)$. This means that every $I^\Delta$ is the toric ideal defined by the point subconfiguration $\Delta$ and $\Proj(\bC[\Delta]/I^\Delta)=\Proj(S/\wt I^\Delta)$ is the projective toric variety defined by $\Delta$. Hence, $\Proj(S/I^w)$ is a \textit{semitoric variety} whose irreducible components are the toric varieties $\Proj(S/\wt I^\Delta)$ for all maximal cells $\Delta$ of $\Theta^*$. In~\cite{Zh} it is proved that this semitoric variety is precisely the reduction of the initial degeneration $\Proj(S/\init_w I)$. More precisely, we have the following.
\begin{theorem}[{\cite[Theorem 3]{Zh}}]\label{Zhu}
If $I$ is toric, then $I^w$ is the radical of $\init_w I$. In particular, $I=I^w$ if $I$ is toric and $\init_w I$ is radical.
\end{theorem}

\section{Categorical limits}

We present categorical counterparts of the results in the previous section. The quotients $S/I_w$ and $S/I^w$ arise as categorical limits, the surjections between these rings and $S/\init_w I$ are the respective mediating morphisms.

\subsection{Diagrams and cones}

We briefly recall basic notions on categorical limits. We follow the treatment in Chapter 5 of \cite{Awodey}.

Let $\bfJ$ and $\bfC$ be categories. A \textit{diagram} of shape $\bfJ$ in $\bfC$ is a functor $\cD \colon \bfJ \to \bfC$. Denote by $\cD_i$ the value of $\cD$ on the object $i$ of $\bfJ$, and $\cD_{\alpha}$ the value of $\cD$ on the morphism $\alpha$.  

A \textit{cone} to a diagram $\cD \colon \bfJ \to \bfC$ consists of an object $C$ in $\bfC$ together with morphisms $c_{i} \colon C\to \cD_i$ such that, for each morphism $\alpha \colon i\to j$ we have $c_{j} = \cD_{\alpha} \circ c_i$. A \textit{limit} of $\cD$ is a cone $(C, \{c_i\})$ such that, if $(C', \{c_{i}'\})$ is another cone to $\cD$, then there is a unique morphism $u\colon C'\to C$ satisfying $c_{i}' = c_{i} \circ u$. We call $u$ the \textit{mediating morphism}. If a limit $(C, \{c_i\})$ of $\cD \colon \bfJ \to \bfC$ exists, then it is unique up to unique isomorphism. In this case, we write $\varprojlim_{\bfJ} \cD_{i}$ to denote the object $C$. 

Dually, colimits are defined in the following way. A \textit{cocone} from a diagram $\cD \colon \bfJ \to \bfC$ is an object $C$ of $\bfC$ together with morphisms $c_{i} \colon \cD_i \to C$ such that, for each morphism $\alpha \colon i\to j$, we have $c_{i} = c_{j} \circ \cD_{\alpha}$. A \textit {colimit} of $\cD$ is a cocone $(C, \{c_i\})$ from $\cD$ such that, if $(C', \{c_{i}'\})$ is another cocone from $\cD$, then there is a unique mediating morphism $v\colon C\to C'$ satisfying $c_{i}' = v \circ c_{i}$. As with limits, if a colimit of $\cD \colon \bfJ \to \bfC$ exists, then it is unique up to unique isomorphism, and we denote the corresponding object by $\varinjlim_{\bfJ} \cD_{i}$.

We are primarily interested in limits and colimits of finite diagrams, i.e.\ diagrams $\cD \colon \bfJ \to \bfC$  where $\bfJ$ is a finite category. (Recall that a $\bfJ$ is a finite category if it has a finite number of objects and there are a finite number of morphisms between each object of $\bfJ$.) Fortunately, there are practical methods to check if a category admits finite limits or colimits. A category has all finite limits if and only if it has pullbacks and a terminal object. Dually, a category has all finite colimits if and only if it has pushouts and an initial object. See Proposition 5.21 and Theorem 5.23 in \cite{Awodey}. 
This applies to the categories we consider in this paper: commutative unital $\bC$-algebras and affine schemes.

\subsection{The colimit}\label{sec:colimit}

Recall the setting and notations of Subsection \ref{sec:lower-bound}.
Denote $R_\Delta=S/I_\Delta$ for a cell $\Delta$ in $\Theta$. Let $\bfJ(\Theta)$ denote the category of cells in $\Theta$ with morphisms given by inclusions of cells. Proposition~\ref{intersection} shows that we have embeddings $R_\Delta\to R_\Gamma$ for inclusions $\Delta\subset\Gamma$. We obtain a diagram $\cR$ of shape $\bfJ(\Theta)$ in the category of commutative unital $\bC$-algebras, that takes $\Delta$ to $R_\Delta$. 

When every $a_e$ lies in some cell of $\Theta$ our categorical interpretation is particularly simple. This is a natural case to consider: for example, this is true when $\cA(I)$ is the vertex set of a convex polytope, as is the case for the Pl\"ucker ideal and many others. Denote $R_w=S/I_w$.

\begin{theorem}\label{thm:colimit}
If every $a_e$ lies in some cell of $\Theta$, then $R_w = \varinjlim_{\bfJ(\Theta)} R_{\Delta}$.
\end{theorem}

This theorem is a special case of Theorem~\ref{thm:colimit-hat} and is not proved separately. We do, however, define the respective cocone by specifying a morphism $c_\Delta\colon R_\Delta\to R_w$ for every $\Delta\in\Theta$. Since $I_w$ contains $I_\Delta$ we have $I_\Delta\subset I_w\cap\bC[\Delta]$. Therefore, $R_\Delta$ surjects onto $\bC[\Delta]/(I_w\cap\bC[\Delta])$ while the latter embeds into $R_w$. We define $c_\Delta$ as the composition of these two maps. We obtain a cocone $ (R_w,\{c_\Delta\})$ from $\cR$.

\begin{example}
    Consider the setting of Example \ref{prinipaltoricex}. Recall that we were working over $E  = \{1,2,3,4\}$ and had determined that there were two maximal subdivisions. We consider the subdivision $\Theta_1$, which has two maximal cells $\Delta_3 = (a_1, a_2,a_3)$ and $\Delta_4 = (a_1,a_2,a_4)$, and five edges $(a_1,a_2)$, $(a_2,a_3)$, $(a_1,a_3)$, $(a_2,a_4)$, and $(a_1,a_4)$. Finally, the vertices in $\Theta_1$ are just the points $a_1,a_2,a_3,$ and $a_4$ since no points coincide. 

    In Example \ref{prinipaltoricex}, we already computed that $I_w = \langle x_1x_2\rangle$ and  saw that $ I_{\Delta_3}  = \langle x_1x_2\rangle  \subseteq \mathbb{C}[\Delta_3]$. We thus compute $R_{\Delta_3} = \bC[x_1,x_2,x_3]/\langle x_1x_2\rangle.$ Analogously, for the cell $\Delta_4$ we obtain $R_{\Delta_4} = \bC[x_1,x_2,x_4]/\langle x_1x_2\rangle.$ 
    
    For every non-maximal cell $\Delta$ other than $(a_1,a_2)$, we have $I_\Delta=0$ and $R_\Delta=\bC[\Delta]$. We also have $I_{(a_1,a_2)}=\al x_1 x_2\ar$ and $R_{(a_1,a_2)} = \bC[x_1,x_2]/ \langle x_1x_2 \rangle$. All arrows in the diagram $\cR$ are seen to be embeddings. Furthermore, $I_\Delta=I_w\cap\bC[\Delta]$ for every $\Delta$, therefore, all $c_\Delta$ are also embeddings. To conclude that $R_w=\varinjlim_{\bfJ(\Theta)} R_{\Delta}$, it remains to note that $R_w$ is generated by the images of the $c_\Delta$.    
    % For the cell $\Delta_{13} = \{a_1,a_3\}$, we obtain $I_{\Delta_{1,3}} = I_w \cap \bC[\Delta_{13}]  = \langle x_1 \rangle $. We then have $R_{\Delta_{1,3}} = \bC[x_1,x_3]/\langle x_1 \rangle \cong \bC[x_3]$, and have an immersion $R_{\Delta_{1,3}} \hookrightarrow R_{\Delta_3}.$ 
    % Analogously, we obtain $R_{\Delta_{1,3}} \cong R_{\Delta_{2,3}}$, and $R_{\Delta_{1,4}}\cong R_{\Delta_{2,4}} \cong \bC[x_4]$. Finally, we have $R_{\Delta_{1,2}} = \bC[x_1,x_2]/ \langle x_1x_2 \rangle.$
\end{example}

In the general case, an adjustment is needed. 
% Consider those variables $x_e$ for which $a_e$ is not contained in any cell of $\Theta$. Let $S_0$ be the ring of polynomials in all such $x_e$.
For a cell $\Delta$ of $\Theta$, let $Q_\Delta$ denote ring of polynomials in all $x_e$ such that $a_e$ is not contained in any cell and $a_e\in\conv\Delta$. 
We denote $\widehat R_\Delta=R_\Delta\otimes_\bC Q_\Delta$. 
When $\Delta\subset\Gamma$, we have embeddings $Q_\Delta\to Q_\Gamma$ and $R_\Delta\to R_\Gamma$ which together provide an embedding $\widehat R_\Delta\to\widehat R_\Gamma$. We obtain a diagram $\widehat \cR$ of shape $\bfJ(\Theta)$ taking $\Delta$ to $\widehat R_\Delta$. Furthermore, by construction, the ideal $I_w$ does not intersect any of the subrings $Q_\Delta\subset S$. This allows us to view every $Q_\Delta$ as a subring of $R_w$. Subsequently, we define morphisms $\widehat c_\Delta\colon \widehat R_\Delta\to R_w$ given by $\widehat  c_\Delta(p\otimes q)= c_\Delta(p)q$. These morphisms form a cocone $(R_w,\{\widehat c_\Delta\})$ from $\widehat \cR$.

\begin{theorem}\label{thm:colimit-hat}
The cocone $(R_{w},\{\widehat c_\Delta\})$ is the colimit of $\cR$, in particular, $R_w=\varinjlim_{\bfJ(\Theta)}\widehat R_{\Delta}$.
\end{theorem}

\begin{proof}
Consider a cocone $(R',\{c'_\Delta\})$ from $\widehat \cR$. For $e\in E$ and $\Delta$ containing $a_e$ let $y_{e,\Delta}\in R_\Delta$ denote the image of $x_e\in\bC[\Delta]$ under the quotient map. If $a_e$ is contained in at least one cell, there is a unique minimal cell $\Delta_e$ containing $a_e$. Since the morphisms $c'_\Delta$ form a cocone, we have $c'_\Delta(y_{e,\Delta}\otimes 1)=c'_{\Delta_e}(y_{e,\Delta_e}\otimes 1)$ for any $\Delta$ containing $a_e$. 

If $a_e$ is not contained in any cell, there is a unique minimal cell $\Delta_e$ whose convex hull contains $a_e$. If $a_e\in\conv\Delta$, we have $c'_\Delta(1\otimes x_e)=c'_{\Delta_e}(1\otimes x_e)$.

Hence, 
\[\psi:x_e\mapsto 
\begin{cases}
 c'_\Delta(y_{e,\Delta}\otimes 1)&\text{ if }a_e\in\Delta,\\
 c'_\Delta(1\otimes x_e)&\text{ if $a_e$ is not contained in any cell and }a_e\in \conv\Delta.
\end{cases} 
\]
gives a well-defined morphism from $S$ to $R'$. By the definition of $y_{e,\Delta}$, the kernel of $\psi$ contains every $I_\Delta$. Hence, $\psi$ is the composition of the quotient map $S\to R_w$ and a morphism $\varphi:R_w\to R'$. 

Now, if $e$ is such that either $y_{e,\Delta}\otimes 1$ or $1\otimes x_e$ is an element of $\widehat R_\Delta$, then $\widehat c_\Delta$ takes this element to the image of $x_e$ under the quotient map $S\to R_w$. Hence, $\varphi$ satisfies $c'_\Delta=\varphi\circ\widehat c_\Delta$ for every $\Delta$. Moreover, $\varphi$ is the unique morphism with this property. Indeed, $c'_\Delta=\varphi\circ\widehat c_\Delta$ determines $\varphi$ uniquely on elements of the image $\widehat c_\Delta(\widehat R_\Delta)$ and such images generate $R_w$. We conclude that $\varphi$ is the unique mediating morphism. 
\end{proof}

\begin{example}
    We return to the setting of Example \ref{doublepointex} and consider the ideal $I= \langle x_1x_2 (x_1-x_2) \rangle$ with a weight $w$ such that $w_1 < w_2$. The point configuration $\mathcal{A}(I)$ contains two points $a_1$ and $a_2$, whereas the  subdivision $\Theta$ only contains one cell: $(a_1)$. Then, we have $R_{(a_1)} = \bC[x_1]/I_{(a_1)} = \bC[x_1]/\langle 0 \rangle$, and since the diagram has one node, its limit is $\widehat R_{(a_1)} = R_{(a_1)}\otimes_\bC Q_{(a_1)} = \bC [x_1] \otimes_\bC \bC [x_2] \cong \bC[x_1,x_2] = R_w.$
\end{example}

Next we define a cocone $(S/\init_w I,\{\xi_\Delta\})$ from $\widehat\cR$. To obtain a map $R_\Delta\to S/\init_w I$ for a given $\Delta$, note that $I_\Delta\subset (\init_w I)\cap\bC[\Delta]$ by Proposition~\ref{initialcontains}, and consider the composition 
\[R_\Delta\twoheadrightarrow \bC[\Delta]/(\bC[\Delta]\cap\init_w I)\hookrightarrow S/\init_w I.\] 
Also, the quotient map $S\to S/\init_w I$ restricts to a map $Q_\Delta\to S/\init_w I$. We define $\xi_\Delta$ as the tensor product of the maps from $R_\Delta$ and $Q_\Delta$ (the second factor is not needed when every $a_e$ lies in some cell of $\Theta$). Theorem~\ref {thm:colimit-hat} now implies the following.
\begin{cor}\label{cor:med-morph-colim}
We have a unique morphism $\alpha\colon R_w\to S/\init_w I$ such that $\alpha\circ\widehat c_\Delta=\xi_\Delta$ for all cells $\Delta$. Furthermore, $\alpha$ coincides with the surjection provided by Theorem~\ref{sumcontained}.
\end{cor}
\begin{proof}
The first claim is immediate from Theorem~\ref {thm:colimit-hat}. For the second claim one needs to check that the surjection $\alpha'$ provided by Theorem~\ref{sumcontained} satisfies $\alpha'\circ\widehat c_\Delta=\xi_\Delta$ for every $\Delta$. It suffices to verify that these two maps coincide on elements of $\widehat R_\Delta$ of the forms $y_{e,\Delta}\otimes 1$ and $1\otimes x_e$ (see proof of Theorem~\ref{thm:colimit-hat}). Indeed, both maps take such an element to the image of $x_e$ in $R_w$.
\end{proof}

% \begin{remark}
% Let $\bfJ'(\Theta)$ denote the subcategory of $\bfJ(\Theta)$ consisting of the maximal cells, the interior cells of codimension 1 and the morphisms between them. We have a diagram $D'$ of shape $\bfJ'(\Theta)$ with $D'_\Delta=R_\Delta$ that is a subdiagram of $D$. It is easily seen that the colimit of $D'$ is simply $(R_w,\{\widehat c_\Delta\}_{\Delta\in\bfJ'(\Theta)})$. In other words, it suffices to consider only the cells in $\bfJ'(\Theta)$ to obtain the ring $R_w$. Indeed, any cocone from $D'$ uniquely extends to a cocone from $D$. This is since every $\widehat R_\Delta$ with $\Delta$ not in $\bfJ'(\Theta)$ is naturally a subring of $\widehat R_\Gamma$ with $\Gamma$ maximal. Furthermore, if $\Delta$ is contained in at least two maximal cells, then it is also contained in an interior cell of codimension 1, this ensures that the cocone is well-defined.
% \end{remark}

\subsection{The limit}

We return to the setting and notations of Subsection~\ref{sec:upper-bound}. Remarkably, in this case the categorical interpretation is more easily stated and does not require us to adjust for points not lying in any cell.

For a cell $\Delta$ of $\Theta^*$ denote the quotient $\bC[\Delta]/I^\Delta$ by $R^\Delta$. For an inclusion of cells $\Delta\subset\Gamma$, Proposition~\ref{prop:projection} provides a surjection $\pi_\Delta^\Gamma:R^\Gamma\to R^\Delta$. The rings $R^\Delta$ and the morphisms $\pi_\Delta^\Gamma$ form a diagram $\cR^*$ of shape $\bfJ(\Theta^*)^\mathrm{op}$, the opposite category of $\bfJ(\Theta^*)$. We have $S/\wt I^\Delta = R^\Delta$ and $I^w\subset \wt I^\Delta$ which provides surjections $ c^\Delta\colon R^w\to R^\Delta$. We obtain a cone $(R^w,\{c^\Delta\})$ to $\cR^*$.
\begin{theorem}\label{thm:limit}
The cone $(R^w,\{c^\Delta\})$ is the limit of $\cR^*$, in particular, $R^w=\varprojlim_{\bfJ(\Theta^*)^\mathrm{op}} R^{\Delta}$.
\end{theorem}
\begin{proof}
Limits in the category of commutative rings (or $\bC$-algebras) have an explicit realization in terms of direct products, see Exercise 10 in \cite[Section~7.6]{DummitFoote}. 
Let $R_0\subset\prod_\Delta R^\Delta$ be the subring of all elements $(r_\Delta)_{\Delta}$ such that for any $\Delta\subset\Gamma$ one has $\pi_\Delta^\Gamma(r_\Gamma)=r_\Delta$. The projections $c_0^\Delta: R_0\to R^\Delta$ form a cone $(R_0,\{c_0^\Delta\})$ to $\cR^*$ that is the limit of the latter.

Since we have a surjection $S\to R^\Delta$ with kernel $\wt I^\Delta$ for every $\Delta$, we have a map $ c \colon S\to\prod_\Delta R^\Delta$ with kernel $I^w$ and image isomorphic to $R^w$. We obtain an embedding $R^w\to\prod_\Delta R^\Delta$, this embedding is the direct product $\prod_\Delta c^\Delta$. To prove the theorem, we are to show that the image $c(S)\simeq R^w$ coincides with $R_0$. From the definition of $c$ it is clear that the image $c(S)$ is contained in $R_0$ and we prove the reverse inclusion.

Choose a cell $\Gamma$. We claim that $I^\Gamma$ is a subspace of $I^w$. Indeed, similarly to the proof of Proposition~\ref{prop:initialcontained}, we have $w'\in w+\cL(I)$ with $w'_e=0$ for $e\in\Gamma$. In particular, $\init_{w'} p=p$ for any $p\in I^\Gamma$. Hence, $I^\Gamma\subset\init_{w'}I=\init_w I\subset I^w$ by Theorem~\ref{thm:intersectioncontains}.

This allows us to define a homomorphism $s_\Gamma:R^\Gamma\to\prod_\Delta R^\Delta$ as follows. For $r_\Gamma\in R^\Gamma$ choose any $p\in\bC[\Gamma]$ that is taken to $r_\Gamma$ by the quotient map, and set $s_\Gamma(r_\Gamma)= c(p)$. This is well-defined because the kernel of $c$ contains $I_\Gamma$.

For $r\in\prod_\Delta R^\Delta$ and a cell $\Delta$ let $r_{\Delta}$ denote the $\Delta$th component of $r$. For $r\in R_0$ consider $r'=r-s_\Gamma(r_\Gamma)$. Note that $s_\Gamma(r_\Gamma)_\Delta$ is equal to 0 if $\Gamma\cap\Delta$ is empty and to $s_{\Gamma\cap\Delta}(r_{\Gamma\cap\Delta})_\Delta$ otherwise. Therefore, $r'_\Delta=0$ for any $\Delta\subset\Gamma$ and, furthermore, $r'_\Delta=0$ whenever $r_\Delta=0$. 

Now, choose $r^0\in R_0$. Let $\Gamma_1,\dots,\Gamma_n$ be the maximal cells of $\Theta^*$. For $i=1,\dots,n$ set $r^i=r^{i-1}-s_{\Gamma_i}(r^{i-1}_{\Gamma_i})$. We see that $r^i_\Delta=0$ for all $\Delta$ that are contained in one of $\Gamma_1,\dots,\Gamma_i$. Hence, $r^n=0$, and we have expressed an arbitrary $r^0\in R_0$ as a sum of elements of $c(S)$, as desired.
\end{proof}

The surjections $\xi^\Delta\colon S/\init_w I\to R^\Delta$ given by Proposition~\ref{prop:initialcontained} form a cone $(S/\init_w I,\{\xi^\Delta\})$ to $\cR^*$. Proposition~\ref{thm:limit} together with the universal property of limits provides

\begin{cor}
We have a unique morphism $\alpha\colon S/\init_w I\to R^w$ such that $ c^\Delta\circ\alpha=\xi^\Delta$ for all cells $\Delta$. Furthermore, $\alpha$ coincides with the surjection provided by Theorem~\ref{thm:intersectioncontains}. 
\end{cor}

\begin{example}
Consider the first case in Example~\ref{ex:upperbound}: $w_1+w_2>w_3+w_4$. All $I^\Delta=0$ and the diagram $\cR^*$ consists simply of the rings $\bC[\Delta]$ and the natural surjections between them. However, the limit of this diagram is the ring $R^w=S/\al x_3x_4\ar$. One way of seeing this is by identifying the limit of $\cR^*$ with the pullback $\bC[\Delta_3]\times_{\bC[x_1,x_2]}\bC[\Delta_4]$, i.e.\ by ignoring all but three nodes of the diagram, cf.\ Subsection~\ref{sec:lim-graph}.
\end{example}

\subsection{Limits, colimits and the adjacency graph}
\label{sec:lim-graph}
The whole category $\bfJ(\Theta)$ of the subdivision $\Theta$ can be quite large and working with the categorical limits and colimits in this section may appear to be challenging.  In this subsection, we describe subcategories of $\bfJ(\Theta)$ which are substantially smaller than $\bfJ(\Theta)$ but yield the same (co)limit. 

The \textit{adjacency graph} to the subdivision $\Theta$ is the graph $\cG(\Theta)$ that has a vertex for each maximal cell of $\Theta$, and two vertices are connected by an edge if and only if they share a codimension-1 face. Thus, the edges of $\cG(\Theta)$ correspond to \textit{interior} codimension-1 cells of $\Theta$, i.e.\ those not contained in a proper face of $\Theta$. We denote by $\bfK(\Theta)$ the full subcategory of $\bfJ(\Theta)$ formed by the maximal cells and the interior codimension-1 cells. The objects of $\bfK(\Theta)$ correspond to vertices and edges of $\cG(\Theta)$, the morphisms correspond to adjacencies between vertices and edges. 

\begin{proposition}
\label{prop:limit-tight-span}
For a category $\bfC$, let $\cD\colon \bfJ(\Theta) \to \bfC$ and $\cD^* \colon \bfJ(\Theta)^{\op} \to \bfC$  be diagrams, and let $\bfJ$ be a subcategory of $\bfJ(\Theta)$ containing $\bfK (\Theta)$. If $\bfC$ admits finite colimits, then the three colimits $\varinjlim_{\bfK(\Theta)} \cD_{\Delta}$, $\varinjlim_{\bfJ} \cD_{\Delta}$ and $\varinjlim_{\bfJ(\Theta)} \cD_{\Delta}$ are mutually isomorphic. Similarly, if $\bfC$ admits finite limits, then $\varprojlim_{\bfJ(\Theta)^{\op}} \cD_{\Delta}^*$,   $\varprojlim_{\bfJ^{\op}} \cD_{\Delta}^*$ and  $\varprojlim_{\bfK(\Theta)^{\op}} \cD_{\Delta}^*$ are mutually isomorphic. 
%and let $\Sigma$ be a subcomplex of $\cT(\Theta)$ that contains $\cG(\Theta)$. 
%Then the (four) natural morphisms
%    \begin{align*}
%    \varinjlim_{\bfK(\Theta)} D_i \xrightarrow{\varphi}  \varinjlim_{\bfJ} D_i \xrightarrow{\psi}  \varinjlim_{\bfJ(\Theta)} D_i \quad \text{and} \quad
%    \varprojlim_{\bfJ(\Theta)} D_i'  \xrightarrow{\psi'}  \varprojlim_{\bfJ} D_i' \xrightarrow{\varphi'} \varprojlim_{\bfK(\Theta)} D_i'
%    \end{align*}
    %\begin{align*}
    %\varinjlim_{\bfJ(\cG(\Theta))^{\op}} D_i \xrightarrow{\varphi}  \varinjlim_{\bfJ(\Sigma)^{\op}} D_i \xrightarrow{\psi}  \varinjlim_{\bfJ(\Theta)} D_i \quad \text{and} \quad
    %\varprojlim_{\bfJ(\Theta)^{\op}} D_i'  \xrightarrow{\psi'}  \varprojlim_{\bfJ(\Sigma)} D_i' \xrightarrow{\varphi'} \varprojlim_{\bfJ(\cG(\Theta))} D_i'
    %\end{align*}
%    are isomorphisms. 
\end{proposition}

\begin{proof}
To show that $\varinjlim_{\bfK(\Theta)} \cD_{\Delta}$ is isomorphic to $\varinjlim_{\bfJ(\Theta)}\cD_{\Delta}$, it suffices to show that $\bfK(\Theta)$ is a final subcategory of $\bfJ(\Theta)$, see~\cite[Section IX.3]{maclane}. This means that for every $\Delta$ in $\bfJ(\Theta)$, the comma category $(\Delta\downarrow\iota)$ is nonempty and connected, where $\iota\colon \bfK(\Theta)\to\bfJ(\Theta)$ is the inclusion functor. This comma category is, up to isomorphism, the full subcategory of $\bfK(\Theta)$ formed by those $\Gamma$ that contain $\Delta$. It is obviously nonempty and it is connected because the subgraph of $\cG(\Theta)$ formed by vertices and edges corresponding to cells that contain $\Delta$ is connected, see e.g.\ \cite[Proposition~6.1]{corey_olarte_2022}. A similar argument shows that $\varinjlim_{\bfK(\Theta)} \cD_{\Delta}$ and $\varinjlim_{\bfJ}\cD_{\Delta}$ are isomorphic.

Dually, $\varprojlim_{\bfJ(\Theta)^{\op}} \cD_{\Delta}^*$ is isomorphic to $\varprojlim_{\bfK(\Theta)^{\op}} \cD_{\Delta}^*$ because the subcategory $\bfK(\Theta)^{\op}$ is initial. This means that the comma categories $(\iota^*\downarrow\Delta)$ (where $\iota^*$ is the inclusion functor) are nonempty and connected. However, $(\iota^*\downarrow\Delta)$ is the dual of the nonempty connected category $(\Delta\downarrow\iota)$. Again, for $\varprojlim_{\bfK(\Theta)^{\op}} \cD_{\Delta}$ and $\varprojlim_{\bfJ^{\op}}\cD_{\Delta}$ the argument is similar.
\end{proof}

Proposition~\ref{prop:limit-tight-span} applies not only to the subdivision $\Theta$ but also to any other regular subdivision, including $\Theta^*$ (in fact, it even holds for non-regular polyhedral subdivisions). In particular, combining Proposition~\ref{prop:limit-tight-span} with Theorems~\ref{thm:colimit-hat} and~\ref{thm:limit}, we see that the rings $R_w$ and $R^w$ can be obtained via (co)limits of subdiagrams of $\cR$ and $\cR^*$.

\begin{remark}
One particularly useful intermediate category $\bfJ$ is formed by the interior cells of $\Theta$. This category may be identified with the face poset of the tight span of $\Theta$: a polyhedral complex that is, roughly speaking, dual to $\Theta$. It has a cell of dimension $k$ for each interior cell of dimension $\dim \Theta - k$, and the face inclusions are opposite those in $\Theta$. See \cite{Herrmann} for a more detailed construction. Note that the adjacency graph of $\Theta$ is the 1-skeleton of the tight span, so $\bfJ$ does indeed contain $\bfK(\Theta)$.   In the Grassmannian setting, these ideas feature prominently in \cite{CoreyLuber}.
\end{remark} 
%While $\bfJ(\cG(\Theta))^{\op}$ is the smallest poset we consider, there are still situations where taking a (co)limit over a slightly larger poset of the tight span is advantageous. This is featured prominently in \cite{CoreyLuber}.
%The adjacency graph, viewed as a $1$-dimensional polyhedral complex, is the 1-skeleton of a larger polyhedral complex that is, roughly speaking, dual to $\Theta$, called the tight span. The \textit{tight span} of $\Theta$ is a polyhedral complex, denoted by $\cT(\Theta)$, that has a cell of dimension $k$ for each cell of dimension $\dim \Theta - k$ which is not contained in a face of $A$, and the face identifications are opposite those in $\Theta$. See \cite{Herrmann} for a more detailed construction. 

\section{Exactness of bounds and subfans}

\subsection{The lower bound}

A natural question raised by the above is for which $w$ the bounds on the initial ideal provided by Theorems~\ref{sumcontained} and~\ref{thm:intersectioncontains} are exact. We show that in both cases the set of such $w$ has a surprisingly nice structure: it is the support of a subfan in the secondary fan.

First, let $\Omega(I)$ denote the set of those $w\in\bR^E$ for which $I_w=\init_w I$. Since $I_w$ is determined by $\subd_w\cA(I)$, one sees that $\Omega(I)$ is a union of relative interiors of cones in the common refinement of $\Grob I$ and $\Sec\cA(I)$. However, one easily verifies a stronger statement using the notion of \textit{graded dimension} of a homogeneous ideal. For such an ideal $J\subset S$ we view $\grdim(J)$ as an infinite sequence of integers indexed by $\bZ_{\ge0}$, with element $\grdim(J)_i$ equal to the dimension of the degree $i$ homogeneous component of $J$. A basic property of initial ideals is that $\grdim(\init_w I)=\grdim(I)$ for any $w$. We will also consider a partial order on graded dimensions by writing $\grdim(J_1)\le\grdim(J_2)$ to indicate that $\grdim(J_1)_i\le \grdim(J_2)_i$ for all $i$.
\begin{proposition}\label{isunion}
$\Omega(I)$ is a union of relative interiors of cones in $\Sec\cA(I)$.
\end{proposition}
\begin{proof}
Choose $w\in\Omega(I)$ and denote $\subd_w\cA(I)=\Theta$. We are to show that $w'\in\Omega(I)$ whenever $w'\in\cone(\cA(I),\Theta)^\circ$. 
Indeed, in this case we have $I_{w'}=I_w=\init_w I$, which implies
\[\grdim(I_{w'})=\grdim(\init_w I)=\grdim(I)=\grdim(\init_{w'} I).\] 
However, Theorem~\ref{sumcontained} provides $I_{w'}\subset\init_{w'} I$, hence equality of graded dimensions is only possible if $I_{w'}=\init_{w'}(I)$.
\end{proof}
%\begin{theorem}\label{issubfan}
%The set $\Omega(I)$ is the support of a subfan of $\Sec\cA(I)$.    
%\end{theorem}

We will show that this proposition can be made substantially stronger. $\Omega(I)$ is not just a union of relative interiors of cones, it is, in fact, the support of a subfan. In the proof of the next lemma, the \textit{initial subspace} $\init_w U$ of a vector subspace $U\subset S$ with respect to a weight $w\in\bR^E$ is the linear span of $\init_w p$ over all $p\in U$.
\begin{lemma}\label{initialofIw}
If $w,w'\in\bR^E$ satisfy $w'\in\cone(\cA(I),\,\subd_w\cA(I))$, then $I_w\subset \init_w(I_{w'})$.
\end{lemma}
\begin{proof}
The hypothesis means that $\subd_w\cA(I)$ refines $\subd_{w'}\cA(I)$. Let $\Delta$ be a maximal cell of $\subd_w\cA(I)$, we have a unique maximal cell $\Delta'$ of $\subd_{w'}\cA(I)$ containing $\Delta$. As in Subsection~\ref {thm:colimit}, let $I_\Delta$ denote the image of $I$ under the surjection $\rho_\Delta \colon S\to\bC[\Delta]$ and define $I_{\Delta'}$ similarly. Since $I_w$ is generated by subspaces of the form $I_\Delta$, it suffices to show that $I_\Delta\subset \init_w(I_{\Delta'})$.

We may consider $w_1\in w+\cL(I)$ such that $(w_1)_e=0$ if $a_e\in\Delta$ and $(w_1)_e>0$ otherwise. Since the subspace $I_{\Delta'}\subset S$ is homogeneous with respect to any weight in $\cL(I)$, we have $\init_{w_1}(I_{\Delta'})=\init_w(I_{\Delta'})$. Now, for nonzero $p\in I_\Delta$ consider an element $p'\in I_{\Delta'}$ such that $\rho_\Delta^{\Delta'}(p')=p$, where $\rho_\Delta^{\Delta'} \colon \bC[\Delta']\to\bC[\Delta]$ takes $x_e$ with $a_e\in\Delta$ to itself and all others to 0. Evidently, $p=\init_{w_1} p'\in\init_w(I_{w'})$.
\end{proof}

We obtain our result as a simple consequence.

\begin{theorem}\label{issubfan}
The set $\Omega(I)$ is the support of a subfan of $\Sec\cA(I)$.    
\end{theorem}
\begin{proof}
We show that if $w\in\Omega(I)$ and $w'\in\cone(\cA(I),\,\subd_w\cA(I))$, then $w'\in\Omega(I)$. This suffices by Proposition~\ref{isunion}. Lemma~\ref{initialofIw} provides
\[
\grdim(I_{w'})=\grdim(\init_w(I_{w'}))\ge\grdim(I_w)=\grdim(\init_w I)=\grdim(\init_{w'} I).
\]
Theorem~\ref{sumcontained} now implies $I_{w'}=\init_{w'}I$.
\end{proof}

\begin{example}
Again, suppose that $I$ is the Pl\"ucker ideal $I_{2,n}$. In this case, $\Sec\cA(I)$ is the secondary fan of (the vertex set of) $\Delta(2,n)$. An important subfan of $\Sec\cA(I)$ is the \emph{Dressian} $\Dr(2,n)$ (see~\cite{HerrmannJensenJoswigSturmfels}): it consists of those cones for which the respective subdivision is matroidal. By the results of \cite{SS}, the support of $\Dr(2,n)$ is $\Trop I$ and Theorem~\ref{Gr2nsame} implies that $\Dr(2,n)$ is contained in the subfan supported on $\Omega(I)$. However, $\Omega(I)$ can be substantially larger than the Dressian. For $n=4$ it is not hard to check that $\Omega(I)=\bR^E$ and has dimension 6 while $\Dr(2,n)$ has dimension 5. For $n=5$, a computation in OSCAR shows that $\Omega(I)$ is a union of 72 of the 102 maximal (i.e.\ 10-dimensional) cones of $\Sec\cA(I)$ while $\Dr(2,n)$ has dimension 7, see the repository~\eqref{githublink}.
\end{example}

\begin{remark}\label{rem:subfanofTGr}
Now, consider arbitrary $1\le k<n$, let $I=I_{k,n}$. The Dressian $\Dr(k,n)$ is, again, defined as the subfan of $\Sec \cA(I)$ parametrizing matroid subdivisions. In~\cite{Tev} it is shown that $\Trop I$ is contained in the support of $\Dr(k,n)$. Moreover, every $\cone(I,J)\subset\Trop I$ is contained in some cone of $\Dr(k,n)$, i.e.\ there is a morphism of fans from the fan structure on $\Trop I$ inherited from $\Grob I$ to $\Dr(k,n)$. This implies that for any $w\in\Omega(I)\cap\Trop I$ all of $\cone(I,\init_w I)$ is contained in $\Omega(I)$. Therefore, $\Omega(I)\cap\Trop I$ is the support of a subfan of the tropical fan (equipped with the Gr\"obner fan structure). For $k=2$, this subfan is the entire tropical fan. For $k=3$ and $n=6$, it can be computed that this subfan contains 30 of the 1035 maximal tropical cones. This computation was done using data for $\Trop I_{3,6}$ in the FAIR file format \cite{FAIR}, provided to the authors by Michael Joswig. For details, see~\eqref{githublink}.
\end{remark}

\subsection{The upper bound}

Let $\Omega^*(I)\subset\bR^E$ be the set of all $w$ for which $I^w=\init_w(I)$. Let $-\Sec\cA(I)$ denote the fan obtained from $\Sec\cA(I)$ by reflection in the origin. It is clear that $\Omega^*(I)$ is a union of relative interiors of cones in the common refinement of $\Grob I$ and $-\Sec\cA(I)$. Similarly to the previous case, we make this stronger in a series of statements.

\begin{proposition}\label{dualisunion}
$\Omega^*(I)$ is a union of relative interiors of cones in $-\Sec\cA(I)$.
\end{proposition}
\begin{proof}
The proof is similar to that of Proposition~\ref{isunion}. If $w\in\Omega^*(I)$ and $-w'$ defines the same subdivision as $-w$, then $I^w=I^{w'}$. Hence, the latter has the same graded dimension as $\init_{w'} I$. In view of Theorem~\ref{prop:initialcontained}, this is only possible if $I^{w'}=\init_{w'} I$.
\end{proof}

\begin{lemma}\label{dualinitialofIw}
If $w,w'\in\bR^E$ satisfy $-w'\in\cone(\cA(I),\,\subd_{-w}\cA(I))$, then $\init_w(I^{w'})\subset I^w$.
\end{lemma}
\begin{proof}
Let $\Delta$ be a maximal cell of $\subd_{-w}\cA(I)$, we have a unique maximal cell $\Delta'$ of $\subd_{-w'}\cA(I)$ containing $\Delta$. Let $\wt I^\Delta$ denote the ideal in $S$ generated by $I\cap\bC[\Delta]$ and all $x_e$ with $a_e\notin\Delta$, define $\wt I^{\Delta'}$ similarly. It suffices to show that $\init_w(\wt I^{\Delta'})\subset \wt I^\Delta$.

There is a weight $w_1\in w+\cL(I)$ such that $(w_1)_e=0$ if $a_e\in\Delta$ and $(w_1)_e<0$ otherwise. We have $\init_{w_1}(\wt I^{\Delta'})=\init_w(\wt I^{\Delta'})$. Consider a nonzero $p\in \wt I^{\Delta'}$. If $p\in\bC[\Delta]$, then \[\init_{w_1} p=p\in I\cap\bC[\Delta]\subset \wt I^\Delta.\] Otherwise, all monomials occurring in $\init_{w_1} p$ contain at least one variable $x_e$ with $a_e\notin\Delta$ and, again, we have $\init_{w_1} p\in\wt I^\Delta$.
\end{proof}

\begin{theorem}\label{dualissubfan}
The set $\Omega^*(I)$ is the support of a subfan of $-\Sec\cA(I)$.    
\end{theorem}
\begin{proof}
If $w\in\Omega^*(I)$ and $-w'\in\cone(\cA(I),\,\subd_{-w}\cA(I))$, Lemma~\ref{dualinitialofIw} provides
\[
\grdim(I^{w'})=\grdim(\init_w(I^{w'}))\le\grdim(I^w)=\grdim(\init_w I)=\grdim(\init_{w'} I).
\]
Theorem~\ref{thm:intersectioncontains} then implies $I^{w'}=\init_{w'}I$.
\end{proof}

\begin{example}
Suppose $I$ is toric. Theorem~\ref{Zhu} implies that $w\in\Omega^*(I)$ if and only if $\init_w I$ is radical. However, the results in~\cite[Chapter 8]{Stu1} show that $\init_w I$ is radical if and only if $-w\in\cone(\cA(I),\Theta)$ for an \emph{unimodular triangulation} $\Theta$, i.e.\ a regular subdivision with every cell formed by the vertices of a unimodular simplex. We conclude that $\Omega^*(I)$ is the union of $-\cone(\cA(I),\Theta)$ over all unimodular triangulations $\Theta$.
\end{example}

\section{The very affine setting}
\label{sec:very-affine}
In this section, we consider the case of very affine varieties. The distinguishing feature of this case is that the categorical statements can be phrased geometrically in terms of diagrams of schemes instead of rings. 

Recall that $X$ is a closed subscheme of $\bP(\bC^{E})$ defined by the homogeneous ideal $I \subset S$. We assume that $X$ meets the dense torus $(\bC^{\times})^{E} / \bC^{\times}$ of $\bP(\bC^{E})$ (or, equivalently, that $I$ is monomial-free) and denote by $X^{\circ}$ the scheme-theoretic intersection of $X$ with this torus. More generally, given a cell $\Delta$ of $\Theta$, we denote by $X_{\Delta}^{\circ}$ the scheme-theoretic intersection of $X$ with the coordinate subtorus $(\bC^{\times})^{\Delta} / \bC^{\times}$ of $\bP(\bC^{E})$. 

Let $\init_w X$ denote the closed subscheme of $\bP(\bC^{E})$ defined by the ideal $\init_w I$ and let $\init_w X^{\circ}$ denote the scheme-theoretic intersection of $\init_w X$ with $(\bC^{\times})^{E} / \bC^{\times}$. We assume that this intersection is non-empty which precisely means that $w\in\Trop I$.
Denote $S^\circ=\bC[x_e^{\pm}]_{e\in E}$ and let $I^\circ$ be the ideal in $S^\circ$ generated by $I$. The coordinate ring of $\init_{w}X^{\circ}$ is the $\bC^{\times}$-invariant subring of $S^\circ/\init_w I^\circ$, where $\init_w I^\circ$ is the ideal in $S^\circ$ generated by $\init_w I$. 

Let us describe the affine coordinate ring of $X_{\Delta}^{\circ}$. First, set $\bC[\Delta]^{\circ} = \bC[x_{e}^{\pm}]_{e\in \Delta}$. Define $I_{\Delta}^{\circ}$ as the ideal of $\bC[\Delta]^{\circ}$ generated by $I_{\Delta}$. 
The affine coordinate ring $R^{\circ}_{\Delta}$ of $X_{\Delta}^{\circ}$ is the $\bC^{\times}$-invariant subring of $\bC[\Delta]^{\circ} / I_{\Delta}^{\circ}$.

%\begin{equation*}
    %R_{\Delta} \otimes \bC[x_e^{\pm} \suchthat e \in \Delta] \cong  
    %\bC[x_e^{\pm} \suchthat e \in \Delta ] 
    %/ (I_{\Delta} \cdot \bC[x_e^{\pm} \suchthat e \in \Delta ]). 
%\end{equation*}

%Following the notation in \S \ref{sec:lower-bound}, given a face $\Delta$ of $\Theta$, we set
%\begin{equation*}
% \bC[\Delta]^{\circ} = \bC[x_{e}^{\pm}]_{e \in \Delta}, \quad  I_{\Delta}^{\circ} =  I_{\Delta} \cdot \bC[\Delta]^{\circ},  \quad  R_{\Delta}^{\circ} = \bC[\Delta]^{\circ} / I_{\Delta}^{\circ}.
%\end{equation*}

%\noindent When $\Delta = \cA(I)$, set $I_{\Delta}^{\circ} = I^{\circ}$. Geometrically, $\bC[\Delta]^{\circ}$ is the affine coordinate ring of the coordinate torus stratum $(\bC^{\times})^\Delta / \bC^{\times}$ whose nonzero coordinates are those indexed by $\Delta$, and  $R_{\Delta}^{\circ}$ is the affine coordinate ring of the scheme-theoretic intersection $x_{\Delta}  = X \cap \bP(\bC^{\times})^\Delta$. 

%The tropicalization of $X^{\circ}$ is 
%\begin{equation*}
%    \Trop(X^{\circ}) = \{w \in \bR^{E}  \suchthat \init_{w} I \text{ does not contain a monomial} \}
%\end{equation*}
% \begin{equation*}
%     \bC[x_e^{\pm} \suchthat e \in E] / (\init_{w} I) \cdot \bC[x_e^{\pm} \suchthat e \in E]
% \end{equation*}
% is not the zero ring, and 
%We also have a subdivision $\subd_{w}\cA(I)$.

\begin{proposition}
\label{prop:very-affine-maps}
    If $\Delta \subset \Gamma$ are cells of $\Theta$, then $X_{\Delta}^{\circ}$ and $X_{\Gamma}^{\circ}$ are nonempty, and the coordinate surjection $(\bC^{\times})^{\Gamma} / \bC^{\times} \to (\bC^{\times})^{\Delta} / \bC^{\times}$  induces a morphism  $X_{\Gamma}^{\circ} \to X_{\Delta}^{\circ}$. 
\end{proposition}
\begin{proof}
    That $X_{\Delta}^{\circ}$ and $X_{\Gamma}^{\circ}$ are nonempty follows from Proposition \ref{initialcontains}. We have, by Proposition \ref{intersection}, an embedding  $R_{\Delta} \to R_{\Gamma}$. This extends to a map  $R^{\circ}_{\Delta} \to R^{\circ}_{\Gamma}$, which defines the map $X_{\Gamma}^{\circ} \to X_{\Delta}^{\circ}$.
\end{proof}
While the homomorphism $R_{\Delta} \to R_{\Gamma}$ is an embedding, the induced map $R_{\Delta}^{\circ} \to R_{\Gamma}^{\circ}$ need not be an embedding. Geometrically, this means that $X_{\Gamma}^{\circ} \to X_{\Delta}^{\circ}$ may not be dominant. One example is in \cite[Section~4.3]{CoreyLuber}. This example is largely inspired by the example in  \cite[Section~5.2]{GelfandGoreskyMacPhersonSerganova} which demonstrates that the decomposition of the Grassmannian  into matroid strata is not a Whitney stratification. 

By Proposition \ref{prop:very-affine-maps}, we may form the finite limit of affine schemes $\varprojlim_{\bfJ(\Theta)^{\op}} X_{\Delta}^{\circ}$ (recall that $\Theta = \subd_{w} \cA(I)$). Note that finite limits exist in the category of affine schemes because this category has fiber products and a terminal object, see \cite[Proposition~5.21]{Awodey}. 

We have the following result, which as a special case recovers \cite[Theorem~1.1]{Cor}. 

\begin{theorem}
\label{thm:very-affine-initial-deg}
    The coordinate surjections $(\bC^{\times})^{E} / \bC^{\times} \to (\bC^{\times})^{\Delta} / \bC^{\times}$ induce morphisms of affine schemes $\init_{w}X^{\circ}~\to~X_{\Delta}^{\circ}$. Therefore, we have a mediating morphism
    \begin{equation}
    \label{eq:init-to-inv}
        \init_{w} X^{\circ} \to \varprojlim_{\bfJ(\Theta)^{\op}} X_{\Delta}^{\circ}. 
    \end{equation}
    If each $a_{e} \in \cA(I)$ lies in a cell of $\Theta$, then this morphism is a closed immersion.
\end{theorem}

\begin{proof}
By Proposition \ref{initialcontains}, we have a ring homomorphism $R_{\Delta} \to S/ \init_{w} I$. By localizing and restricting to $\bC^{\times}$-invariant subrings, this produces a ring homomorphism $R_{\Delta}^{\circ} \to \bC[\init_{w}X^{\circ}]$. Applying the $\Spec$ functor yields the desired morphism $\init_{w}X^{\circ} \to X_{\Delta}^{\circ}$.

%The map $\init_{w}X^{\circ} \to X_{\Delta}^{\circ}$ is defined by $R_{\Delta}^{\circ} \to \bC[\init_{w}X^{\circ}]$, which is well defined by Proposition \ref{initialcontains}. 
    
As described in the proof of Proposition \ref{prop:very-affine-maps}, when $\Delta \subset \Gamma$ are cells of $\Theta$, the map $X_{\Gamma}^{\circ} \to X_{\Delta}^{\circ}$ is defined by the ring homomorphism $R_{\Delta}^{\circ} \to R_{\Gamma}^{\circ}$. 
We may therefore form the colimit $\varinjlim_{\bfJ(\Theta)} R_{\Delta}^{\circ}$. The $\Spec$ functor is a contravariant right-adjoint functor (its left-adjoint is the global sections functor), and so it takes colimits to limits, see  \cite[Proposition~9.14]{Awodey}. Therefore, the affine coordinate ring of $\varprojlim_{\bfJ(\Theta)^{\op}} X_{\Delta}^{\circ}$ is  $\varinjlim_{\bfJ(\Theta)} R_{\Delta}^{\circ}$. For the last statement, it suffices to show that the mediating morphism $\varinjlim_{\bfJ(\Theta)} R_{\Delta}^{\circ} \to \bC[\init_{w}X^{\circ}]$ is surjective.

%Now, choose an element $e_0\in E$ and let $y_e$ denote the image of $x_e x_{e_0}^{-1}$ in $S^\circ/\init_w I^\circ$. The elements $y_e$ lie in $\bC[\init_{w}X^{\circ}]$ and generate this ring. If $a_e\in \cA(I)$ lies in the cell $\Delta$ of $\Theta$, then $y_e \in \bC[\init_{w}X^{\circ}]$ lies in the image of $R_{\Delta}^{\circ} \to \bC[\init_{w}X^{\circ}]$. So if each $a_e\in \cA(I)$ lies in a cell of $\Theta$, then the map $\varinjlim_{\bfJ(\Theta)} R_{\Delta}^{\circ} \to\bC[\init_{w}X^{\circ}]$ is surjective, and so the map in Equation \eqref{eq:init-to-inv} is a closed immersion. 

%Finally, consider the last statement. 
The morphisms $R_{\Delta} \to S/ \init_{w} I$ provide localized morphisms $\bC[\Delta]^{\circ} / I_{\Delta}^{\circ} \to S^{\circ} / \init_{w} I^{\circ}$ and a mediating morphism $\varinjlim_{\bfJ(\Theta)} \bC[\Delta]^{\circ} / I_{\Delta}^{\circ} \to S^{\circ} / \init_{w} I^{\circ}$.
Let $y_e$ denote the image of $x_e$ in $S^\circ/\init_w I^\circ$. If $a_e\in \cA(I)$ lies in the cell $\Delta$ of $\Theta$, then $y_e$ lies in the image of $\bC[\Delta]^{\circ} / I_{\Delta}^{\circ} \to S^{\circ} / \init_{w} I^{\circ}$. So if each $a_{e}$ lies in a cell of $\Theta$, then $\varinjlim_{\bfJ(\Theta)} \bC[\Delta]^{\circ} / I_{\Delta}^{\circ} \to S^{\circ} / \init_{w} I^{\circ}$ is surjective because the $y_e$ and their inverses generate $S^{\circ} / \init_{w} I^{\circ}$. Since this ring homomorphism is $\bC^{\times}$-equivariant, it restricts to a surjective ring homomorphism $\varinjlim_{\bfJ(\Theta)} R_{\Delta}^{\circ} \to \bC[\init_{w}X^{\circ}]$.
% , and so the map in Equation \eqref{eq:init-to-inv} is a closed immersion. 
\end{proof}

Finally, we show that the limit considered in the above theorem has an explicit description in terms of sums of ideals, similarly to the projective setting discussed in Subsections~\ref{sec:lower-bound} and~\ref{sec:colimit}. Let $I_w^\circ\subset S^\circ$ be the ideal generated by the subspaces $I_\Delta^\circ\subset\bC[\Delta^\circ]\subset S^\circ$ for all cells $\Delta$ of $\Theta$. We denote by $R_w^\circ$ the $\bC^\times$-invariant subring of $S^\circ/I_w^\circ$ and set $X_w^\circ=\Spec R_w^\circ$.
\begin{theorem}\label{thm: limit is Xw}
If each $a_{e} \in \cA(I)$ lies in a cell of $\Theta$, then
\[\varprojlim_{\bfJ(\Theta)^{\op}} X_{\Delta}^{\circ}=X_w^\circ.\]
\end{theorem}
\begin{proof}
Let $\cR^\circ$ denote the diagram of shape $\bfJ(\Theta)$ formed by the rings $\bC[\Delta]^{\circ} / I_{\Delta}^{\circ}$. It suffices to show that
\[\varinjlim_{\bfJ(\Theta)} \bC[\Delta]^{\circ} / I_{\Delta}^{\circ}=S^\circ/I_w^\circ,\]
we may then pass to $\bC^\times$-invariants and apply $\Spec$. The maps $c_\Delta\colon R_\Delta\to R_w$ considered in Subsection~\ref{sec:colimit} can be localized to obtain maps $c_\Delta^\circ\colon \bC[\Delta]^{\circ} / I_{\Delta}^{\circ}\to S^\circ/I_w^\circ$. 
% That is since $\bC[\Delta]^{\circ} / I_{\Delta}^{\circ}$ is the localization of $R_\Delta$ and $S^\circ/I_w^\circ$ is the localization of $R_w$, in both cases by the quotient-map images of the $x_e$.

For every $\Delta$ we have a localization map $\iota_\Delta\colon R_\Delta\to\bC[\Delta]^{\circ} / I_{\Delta}^{\circ}$. Consider a cocone $(R',\{c_\Delta\})$ from $\cR^\circ$. It induces a cocone $(R',\{c_\Delta\circ\iota_\Delta\})$ from the diagram $\cR$ considered in Subsection~\ref{sec:colimit}. Since $R_w$ is the colimit of $\cR$, we have a unique mediating morphism $v:R_w\to R'$. The image of every $y_e$ under $v$ is invertible where $y_e$ denotes the image of $x_e$ in $R_w$.  By the universal property of localization, we can extend $v$ to a morphism $v^\circ\colon S^\circ/I_w^\circ\to R'$, which is the desired mediating morphism. The uniqueness of this mediating morphism is also evident: its composition with the localization map $R_w\to S^\circ/I_w^\circ$ must coincide with $v$.
\end{proof}

Of course, situations in which the mediating morphism obtained in Theorem~\ref{thm:very-affine-initial-deg} is an isomorphism, are of particular interest. It is easily seen that this holds whenever the analogous statement holds in the projective setting. Indeed, $\bC[\init_w X^\circ]$ is the $\bC^\times$-invariant subring of the localization of $S/\init_w I$ while $R_w^\circ$ is the $\bC^\times$-invariant subring of the localization of $R_w$. We have seen that the latter two rings coincide with the colimits of the respective diagrams, this provides
\begin{cor}
\label{cor:proj-to-va}
Suppose that each $a_e$ lies in some cell of $\Theta$. If the mediating morphism in Corollary~\ref{cor:med-morph-colim} is an isomorphism, then so is the mediating morphism in Theorem~\ref{thm:very-affine-initial-deg}.
\end{cor}

\begin{remark}
The converse of the above is not true. By~\cite{Cor}, when $I=I_{3,6}$ and $X^\circ$ is the very affine part of $\Gr(3,6)$, the closed immersion in Theorem~\ref{thm:very-affine-initial-deg} is an isomorphism for any $w\in \Trop I$. However, as discussed in Remark~\ref{rem:subfanofTGr}, there are $w\in \Trop I$ for which the inclusion $I_w\subset\init_w I$ is proper.
\end{remark}

\begin{remark}
For simplicity, we state the last part of Theorem~\ref{thm:very-affine-initial-deg} and Theorem~\ref{thm: limit is Xw} only in the case of all $a_e$ lying in some cell. However, it is natural to expect an extension of these results to the general case. In analogy with Theorem~\ref{thm:colimit-hat}, such a generalization would replace the schemes $X_\Delta^\circ$ by the products $X_\Delta^\circ\times(\bC^\times)^{F}$ where $F = \{e \suchthat a_e\in\conv\Delta,\, a_e\notin\Delta\}$. 
\end{remark}

\bibliographystyle{plainurl}
\bibliography{refs.bib}

\appendix

\section{Grassmannians and matroid strata}
\label{appendix}

In this section, we consider the Grassmannian $\Gr(k,n)$. 
Given $w \in \Trop I_{k,n}$, the regular subdivision $\Theta = \subd_{w}\cA(I_{k,n})$ is matroidal in the sense that each cell of $\Theta$ is the point configuration of a rank $k$ matroid on $[n]$. In this case, every $X_{\Delta}^{\circ}$ appearing in Theorem \ref{thm:very-affine-initial-deg} is a matroid stratum of $\Gr(k,n)$. 
Moreover, by Theorem \ref{Gr2nsame} and Corollary~\ref{cor:proj-to-va}, the mediating morphism in Theorem \ref{thm:very-affine-initial-deg} map is an isomorphism for $k=2$. It is also an isomorphism for all $w\in \Trop I_{k,n}$ for $k=3$ and $n \leq 8$ \cite{Cor, CoreyLuber}.

In this section, we give a new infinite family of initial degenerations $\init_{w}\Gr(k,n)^{\circ}$ where the mediating morphism is an isomorphism, motivated by the class of inductively connected matroids as defined in \cite{LiwskiMohammadi}.

\subsection{Inductively connected paving matroids}
We rapidly review some basic matroid theory. We follow the treatment in \cite{Oxley} and refer the reader to this text for details. 
A matroid (of rank $k$) is a pair $M = (\cE(M), \cB(M))$ where $\cE(M)$ is a finite set and $\cB(M)$ is a subset of $\binom{\cE(M)}{k}$ satisfying the \textit{basis exchange axiom}: for each pair of distinct  $A, B\in \cB$ and $a\in A$, there is a $b\in B$ such that $(A\setminus \{a\}) \cup \{b\}$ lies in $\cB(M)$. Here $\cE(M)$ is called the \textit{ground set} of $M$, an element of $\cB(M)$ is called a \textit{basis} of $M$.  An \textit{independent set} of $M$ is a subset of $\cE(M)$ that is contained in a basis of $M$; the set of independent sets of $M$ is denoted by $\cI(M)$. A \textit{circuit} of $M$ is a minimal dependent set. The \textit{rank} $\rk_{M}(X)$ of $X\subset \cE(M)$ is the size of the largest independent subset of $X$, and the rank of $M$ is the rank of its ground set. %A matroid is completely characterized by its set of independent sets, its set of circuits, or its rank function. 

Given a subset $X\subset \cE(M)$, the \textit{restriction} of $M$ to $X$ is the matroid, denoted by $M|X$, whose independent sets are
\begin{equation*}
    \cI(M|X) = \{ I \subseteq X \suchthat I\in \cI(M)\}.
\end{equation*}
An element $e\in \cE(M)$ is a \textit{coloop} of $M$ if it is contained in every basis of $M$. 
A \textit{flat} of $M$ is a subset $F\subset \cE(M)$ such that $\rk_{M}(F \cup \{e\}) = \rk_{M}(F) + 1$ for each $e\in \cE(M)\setminus F$. A \textit{cyclic flat} of $M$ is a flat $F$ such that $M|F$ has no coloops. Denote by $\cZ^{\ell}(M)$ the set of rank $k-\ell$ cyclic flats of $M$. Given $e\in \cE(M)$ let $\cZ^{\ell}(M, e)\subset \cZ^{\ell}(M)$ be the subset consisting of elements containing $e$. 

The matroid $M$ is \textit{paving}  if each circuit has at least $k$ elements. A restriction $M|X$ of a paving matroid $M$ is a paving matroid, since circuits of the restriction $M|X$ are circuits of $M$ and the rank of $M|X$ is no greater than $k$.

If $M$ is paving, then it is \textit{inductively connected} if there is no nonempty subset $X \subseteq \cE(M)$ such that 
\begin{equation*}
   \{e \in X \suchthat  |\cZ^{1}(M|X, e)| > 2 \} = X. 
\end{equation*}
See \cite[Lemma~4.7]{LiwskiMohammadi}. 

\subsection{Matroid strata and inverse limits}
Denote by $\varepsilon_{1}, \ldots, \varepsilon_{n}$ the standard basis of $\bR^{n}$. Given a subset $B\subset [1, n]$, let 
\begin{equation*}
    a_{B} = \sum_{i\in B} \varepsilon_i
\end{equation*} 
The point configuration $A_{k,n} = \cA(I_{k,n})$, considered in Example \ref{hypersimplexex}, is explicitly:
\begin{equation*}
    A_{k,n} = (a_{B} \suchthat B\subset[1, n],\ |B| = k ).
\end{equation*}
Given a matroid $M$ of rank $k$ on $[1, n]$, its ($\cB(M)$--labeled) point configuration is 
\begin{equation*}
    A_{M} = (a_{B})_{B\in \cB(M)}.
\end{equation*}
The matroid $M$ is \textit{connected} if $\dim A_{M} = n-1$.

Now, suppose that $M$ is a connected, paving and $\bC$-realizable matroid of rank $k$ on $[1,n]$. Let $w(M)\in \bR^{\binom{n}{k}}$ be the corank vector of $M$, i.e.
\begin{equation*}
    w(M)_{B} = n-\rk_{M}(B).
\end{equation*}
Set $\Theta = \subd_{w(M)} A_{k,n}$. Note that $A_M$ is a maximal cell of $\Theta$. 

Given a subset $F\subset [1,n]$, define the rank $k$ matroids $N_{F}$ and $N_{F}'$ on $[1,n]$ by
\begin{equation*}
    \cB(N_{F}) = \{B\in \textstyle{\binom{[1, n]}{k}} \suchthat |B\cap F| \geq k-1\}, \quad 
    \cB(N_{F}') = \{B\in \textstyle{\binom{[1, n]}{k}} \suchthat |B\cap F| = k-1\}.
\end{equation*}
By \cite[Proposition~7.5]{CoreyLuber-SMRS}, the dual graph $\cG(\Theta)$ is a star-shaped graph whose central node corresponds to $A_{M}$, it has a leaf-vertex corresponding to $A_{N_{F}}$ for each $F\in \cZ^{1}(M)$ and the edge connecting this vertex to the central node corresponds to $A_{N_{F}'}$. In particular, $\Theta$ is \textit{matroidal} in the sense that each cell $\Delta$ of $\Theta$ is equal to $A_{M_{\Delta}}$ for a matroid $M_{\Delta}$ of rank $k$ on $[1,n]$. Because $\Theta$ is the corank subdivision of a connected $\bC$-realizable matroid, $M_{\Delta}$ is $\bC$-realizable for each cell $\Delta$ of $\Theta$, see~\cite[Proposition 34 and Theorem 35]{JoswigSchroter}.

In view of realizability, the affine scheme $X_{\Delta}^{\circ}$ appearing in Section \ref{sec:very-affine} can be interpreted as the \textit{matroid stratum} $\Gr(M_{\Delta})^{\circ}\subset \Gr(k,n)$. As discussed in Subsection~\ref{sec:lim-graph}, the diagram of shape $\bfJ(\Theta)^\op$ formed by the $\Gr(M_{\Delta})^{\circ}$, $\Delta\in\Theta$ has a subdiagram of shape $\bfK(\Theta)^\op$. In view of the above, this subdiagram is formed by $\Gr(M)^\circ$ together with the schemes $\Gr(N_F)^\circ$ and $\Gr(N'_F)^\circ$ for all $F\in\cZ^{1}(M)$. By Proposition~\ref{prop:limit-tight-span},
\[\varprojlim_{\bfJ(\Theta)^\op}\Gr(M_\Delta)^{\circ} = \varprojlim_{\bfK(\Theta)^\op}\Gr(M_\Delta)^{\circ}.\]

Set
\begin{equation*}
    Y = \prod_{F\in \cZ^1(M)} \Gr(N_F)^{\circ} \quad \text{and} \quad  
    Y' = \prod_{F\in \cZ^1(M)} \Gr(N_F')^{\circ}.
\end{equation*}
The morphisms $\Gr(N_{F})^{\circ} \to \Gr(N_{F}')^{\circ}$ naturally provide a morphism $Y\to Y'$. Similarly, the morphisms $\Gr(M)^{\circ} \to \Gr(N_{F}')^{\circ}$ provide a morphism $\Gr(M)^{\circ} \to Y'$. We claim that the limit $\varprojlim_{\bfK(\Theta)^\op}\Gr(M_\Delta)^{\circ}$ is isomorphic to the fiber product $\Gr(M)^{\circ} \times_{Y'} Y$. Indeed, there is a clear bijection between the cones to the subdiagram of shape $\bfK(\Theta)^\op$ and cones to the three-node diagram corresponding to the above fiber product. Since the limit of a diagram is the terminal object in the category of cones to that diagram, the limits of the two diagrams must coincide. In summary, we have an isomorphism 
\begin{equation*}
\varprojlim_{\bfJ(\Theta)^{\op}}\Gr(M_{\Delta})^{\circ} \cong \Gr(M)^{\circ} \times_{Y'} Y. \;% \text{where} \; 
    %Y = \prod_{F\in \cZ^1(M)} \Gr(N_F)^{\circ} \quad \text{and} \quad  
    %Y' = \prod_{F\in \cZ^1(M)} \Gr(N_F')^{\circ}.
\end{equation*}
By \cite[Theorem~35]{JoswigSchroter}, we have that $w(M)\in\Trop I_{k,n}$, and so we may form the initial degeneration $\init_{w(M)} \Gr(k,n)^{\circ}$.
%{\color{red} Mention that $w(M)\in\Trop I_{k,n}$ which lets us consider the initial degeneration.}

\begin{theorem}
%Suppose $M$ is a an inductively connected paving matroid $M$ whose dependent hyperplanes are $\cH$. 
Let $M$ be a connected paving matroid. We have an isomorphism  
    \begin{equation*}
        \init_{w(M)} \Gr(k,n)^{\circ} \cong \Gr(M)^{\circ} \times_{Y'} Y.
    \end{equation*}
Furthermore, $\init_{w(M)} \Gr(k,n)^{\circ}$ is smooth and irreducible. 
\end{theorem}

\begin{proof}
    To show that the the closed immersion
    
\begin{equation*}
    \init_{w}\Gr(k,n)^{\circ} \to \varprojlim_{\bfJ(\Theta)^{\op}} \Gr(M_{\Delta})^{\circ}
\end{equation*}
    is an isomorphism when specialized to  $w(M)$, it suffices to show that $\Gr(M)^\circ \times_{Y'} Y$ is smooth and irreducible of dimension $k(n-k)$. The morphism $\Gr(M)^\circ \times_{Y'} Y \to \Gr(M)^{\circ}$ is smooth and surjective by \cite[Proposition~7.6]{CoreyLuber-SMRS}. This proposition also asserts that  
\begin{equation*}
\label{eq:dim-fiber-product}
    \dim(\Gr(M)^{\circ} \times_{Y'} Y) = \dim(\Gr(M)^{\circ}) + \sum_{F\in \cZ^1(M)} (|F| - k + 1). 
\end{equation*}
Now by \cite[Theorem~A]{LiwskiMohammadi}, $\Gr(M)^{\circ}$ is smooth and irreducible with
\begin{equation*}
    \dim(\Gr(M)^{\circ}) = k(n-k) - \sum_{F\in \cZ^1(M)} (|F| - k + 1).
\end{equation*}
Combining these two results, we conclude that $\Gr(M)^{\circ} \times_{Y'} Y$ is smooth and irreducible of dimension $k(n-k)$, as required.
\end{proof}

% {\color{red} Can we say anything about the projective case?}

% {\color{blue} My guess is that the saturation of $I_w$ equals the intersection $I_{w}^{\circ}$ with the polynomial ring $S$. But idk how to prove this.}

\section{Obtaining the point configuration $\cA(I)$ computationally}\label{appendixB}

In this short appendix, we describe how to obtain the point configuration $\cA(I)$ from an ideal $I$ using computational methods. For the purposes of this section, we set $E=\{1,\dots,m\}$ and $I\subset \bC[x_1,\ldots,x_m]=S$. 
Recall that $\cL(I) \subset \bR^{m}$ is the lineality space of $\Grob I$. We say that $p\in S$ is \textit{$\cL(I)$--homogeneous} if it is homogeneous with respect to every $w\in\cL(I)$. While $I$ is generated by $\cL(I)$--homogeneous polynomials, an arbitrary generating set may contain elements that are not of this form. This issue is resolved by the following proposition.
%We can obtain $\cL(I)$ once we have a reduced Gr\"obner basis of $I$ relative to a single specific term order, as we see in the following proposition. 

\begin{proposition}\label{prop:redGrob}
    If $\cG$ is a reduced Gr\"obner basis of $I$ relative to a fixed (Artinian) monomial order $<$, then each $f\in \cG$ is $\cL(I)$--homogeneous. 
\end{proposition}

\begin{proof}
Suppose $f\in \cG$ is not $\cL(I)$--homogeneous. We may decompose $f$ into the sum of $\cL(I)$--homogeneous components, all of which lie in $I$. Choose such a component $g$ that does not contain the leading term $\init_< f$. The leading term $\init_< g$ lies in the initial ideal $\init_< I$  and must be divisible by the leading term of some other element of $\cG$, which contradicts reducedness.
    %Indeed, if some element g is not homogeneous, then it decomposes into a sum of several nonzero homogeneous components w.r.t w, each of which lies in I. Consider any such component g' that does not contain the leading term of g. The leading term of g' lies in the initial ideal and hence must be divisible by the leading term of some other element of the Gr\"obner basis, which contradicts reducedness.
\end{proof}

Now, given a reduced Gr\"obner basis $\cG=\{f_1,\dots,f_n\}$ of $I$, the computation of $\cL(I)$ reduces to finding a matrix kernel. Let $f_i$ contain $l_i$ nonzero monomials with exponent vectors $v^i_1,\dots,v^i_{l_i}\in\bZ_{\ge0}^m$. Let $N$ be a matrix with $m$ columns, whose rows are the $l_1+\dots+l_n-n$ vectors $v^i_1-v^i_j$ with $i\in[1,n]$ and $j\in[2,l_i]$.

\begin{proposition}\label{prop:matrixN}
For $w\in\bR^m$, one has $w\in\cL(I)$ if and only if $Nw=0$.      
\end{proposition}
\begin{proof}
By our definition of $N$, the condition $Nw=0$ is equivalent to every $f_i$ being homogeneous with respect to $w$. Since $I$ is generated by $\cG$, we deduce that $w\in\cL(I)$ whenever $Nw=0$. Conversely, Proposition~\ref{prop:redGrob} shows that if $w\in\cL(I)$, then $Nw=0$.
\end{proof}

Given a matrix $N$ as in the above proposition, standard computer algebra tools let us find a matrix $M$ whose columns form a basis of $\ker N=\cL(I)$. The next statement allows us to pass to $\cA(I)$.

\begin{proposition}\label{prop:Afrommatrix}
% Let $P$ be a matrix with $m$ columns whose rows span $\cM(I)$. Form a matrix $Q$ with $m$ rows whose columns span the kernel of $P$, which is $\cL(I)$. The rows of $P$ form a point configuration in the affine equivalence class of $\cA(I)$. 
Let $M$ be a real matrix with $m$ rows whose columns form a basis of $\cL(I)$. Each row of $M$ may be viewed as a point in $\cL(I)^*$, written in the dual basis. The obtained configuration of $m$ points in $\cL(I)^*$ is precisely $\cA(I)$. 
\end{proposition}
\begin{proof}
For $j\in[1,\dim\cL(I)]$, let $v_j\in\bR^m$ be the $j$th column of $M$. As in Subsection~\ref{subspaceconfig}, for $i\in[1,m]$, let $u^{\cL(I)}_i\in\cL(I)^*$ denote the restriction of the $i$th coordinate function to $\cL(I)$. The proposition asserts that $u^{\cL(I)}_i(v_j)=M_{i,j}$ for all $i,j$, which is by construction.
\end{proof}

To summarize, given an ideal $I$, the point configuration $\cA(I)$ can be found as follows. One first applies a Gr\"obner basis computation to obtain a reduced Gr\"obner basis $\cG$. Next, the kernel of the matrix $N$ considered in Proposition~\ref{prop:matrixN} is computed. The result will be a matrix whose columns form a basis of $\cL(I)$ and, by Proposition~\ref{prop:Afrommatrix}, the rows of this matrix will form $\cA(I)$. An implementation of this algorithm is provided in the repository~\eqref{githublink}.

%\begin{remark}
%    We provide a  matrix representation of $\cA(L)$ which may be useful for concrete computations. Let $k = \dim(L)$. Let $M$ be a $|E|\times k$ matrix whose rows are labeled by $E$ and whose columns form an ordered basis of $L$. Then the rows of $M$ form an $E$-labeled point configuration in the affine equivalence class of $\cA(L)$. 
%\end{remark}

\end{document}